\documentclass{article}
\usepackage[utf8]{inputenc}
\usepackage[T1]{fontenc}
\usepackage{graphicx}
\usepackage{amsmath}
\usepackage{amsthm}
\usepackage{amssymb}
\usepackage{mathtools}
\usepackage{enumitem}
\usepackage{cancel}
\usepackage[margin=1in, letterpaper]{geometry}
\usepackage{setspace} 
\usepackage{titlesec, titletoc} 
\usepackage[numbers,sort]{natbib}
\usepackage{url} 
\mathtoolsset{showonlyrefs}
\usepackage[colorlinks,  citecolor=green, linkcolor = black, urlcolor = cyan]{hyperref} 
\theoremstyle{plain}
\newtheorem{theorem}{Theorem}
\newtheorem{proposition}[theorem]{Proposition}

\newtheorem{lemma}[theorem]{Lemma}

\theoremstyle{definition}
\newtheorem{definition}[theorem]{Definition}

\theoremstyle{remark}
\newtheorem{remark}[theorem]{Remark}

\newcommand{\R}{\mathbb{R}}

\newcommand{\p}{\mathbf{p}}
\newcommand{\q}{\mathbf{q}}
\newcommand{\x}{\mathbf{x}}
\renewcommand{\u}{\mathbf{u}}
\newcommand{\bftheta}{\boldsymbol{\theta}}
\newcommand{\vv}{\mathbf{v}}
\newcommand{\bfphi}{\boldsymbol{\phi}}
\newcommand{\bfvarphi}{\boldsymbol{\varphi}}
\newcommand{\Z}{\mathbf{Z}}
\newcommand{\bfnu}{\boldsymbol{\nu}}
\newcommand{\w}{\mathbf{w}}
\newcommand{\bfkappa}{\boldsymbol{\kappa}}
\newcommand{\bfxi}{\boldsymbol{\xi}}
\newcommand{\bfeta}{\boldsymbol{\eta}}
\newcommand{\bftau}{\boldsymbol{\tau}}

\title{Curvature flow of networks with triple junction drag}
\author{Yuchuan Yang\thanks{Department of Mathematics, University of Michigan. email: \texttt{yuchuan@umich.edu}} \and Selim Esedo\={g}lu\thanks{Department of Mathematics, University of Michigan. email: \texttt{esedoglu@umich.edu}}}
\date{}

\begin{document}

\maketitle

\begin{abstract}
    We consider a PDE system that describes curvature motion of networks with a dynamic boundary condition known as triple junction drag. This model arises in the study of grain boundary evolution in polycrystalline materials. In this system, the surface tension coefficients depend on the crystallographic orientations of the grains, which are allowed to rotate. We prove existence and uniqueness of solutions to this system in the parabolic H\"{o}lder class $C^{2+\alpha,1+\alpha/2}$. Moreover, we extend our existence result to accommodate a wider class of initial networks by relaxing the compatibility conditions on the angles and curvatures at the triple junction. As an important by-product of our result, we demonstrate the possibility of a new type of topological change during the evolution of the network. We also revisit the question of stability of stationary networks and how it is affected by the choice of surface tensions.
\end{abstract}

\section{Introduction}
Motion by mean curvature of networks is a multiphase interfacial evolution that plays a prominent role in many applications.
One of the best known such is grain boundary motion in polycrystalline materials.
This evolution can be understood at a formal level as gradient flow for a cost function associated with the microstructure of the material that consists of the total surface area of interfaces separating neighboring grains (single crystal pieces) from one another -- an internal energy of the material that is dissipated during manufacturing processes such as annealing.
The metric of this formal gradient flow is the $L^2$ norm of the perturbation to the interfaces in the normal direction.
One of the challenging aspects of this dynamics are the free boundaries known as junctions along which three or more interfaces meet.
The evolution of junctions is part of the unknown in the problem; it is hoped that specifying reasonable boundary conditions (e.g. equilibrium boundary conditions known as the Herring angle condition \cite{herring} that specify the angles formed by the meeting of the interfaces) uniquely determines the evolution, at least starting from a good initial condition.

Another challenge is topological changes in the network that inevitably occur eventually, even after starting from good initial data, e.g. in the form of junction -- junction collisions.
This makes it particularly interesting to study weak solutions of the evolution that allow flowing through such singular events.
Nevertheless, a first step is typically a short time, classical existence and uniqueness result starting from smooth enough initial data that satisfy some compatibility with the junction conditions; the interval of existence guaranteed is bound to be well short of the occurrence of any singularity or topological changes.

When Herring angle conditions \cite{herring} are imposed at the free boundaries known as junctions, such a short time existence result was given in two dimensions in the well-known early work of Bronsard \& Reitich \cite{bronsardreitich}.
One of the problems we study in this paper is the same short time existence question, but with different free boundary conditions that arise from using a different metric in the formal gradient flow formulation of the dynamics: The difference in the metric takes the form of an additional penalty on perturbations to the location of the junctions, and results in their typically slower evolution in response to the motion of the interfaces attached to them.
This related but different evolution, known as curvature motion of a network with {\em triple junction drag}, appears in models of microstructural evolution for nanocrystalline materials, see e.g. \cite{GOTTSTEIN2000397}.

Recently, there has been a resurgence of interest in curvature flow with triple junction drag as a possible explanation for the discrepancy between simulations with Herring angle condition and newly available experimental measurements of microstructure evolution \cite{rohrer_annual,peng_2022}.
Some of these studies also explore other effects, such as grain rotation that leads to dynamic surface tensions \cite{esedoglu_rotation,ep1,Kagayaetal2024}.
The main purpose of the present paper is to comment on these recent studies, especially \cite{ep1}, in the following manner:
\begin{enumerate}
\item Recent study \cite{ep1} asserts a short time existence claim for curve shortening flow with triple junction drag and dynamic surface tensions under the significant simplification that replaces all curves in the network with line segments, reducing the model to an ODE system.
  We state an existence and uniqueness result for the coupled PDE system that describes the original model where interfaces are parametrized curves, in the spirit of \cite{bronsardreitich}.
  However, the fixed point iteration scheme is {\em quite different} than the one in \cite{bronsardreitich}.
\item Using techniques from \cite{pluda}, we establish well-posedness in more general function spaces that can accommodate a wider class of initial curves than in both \cite{bronsardreitich} and \cite{pluda}.
In particular, there are no requirements on the \emph{derivatives} (e.g. angles, curvatures, etc.) of the initial curves at the triple junction.
This might be relevant for restarting flows beyond topological changes \cite{irregularnetworks}, such as a neighbor switching event immediately after which typical compatibility conditions of the Herring case are necessarily violated at newly formed triple junctions.

\item We highlight that additional types of topological changes in the network are enabled by replacing the Herring condition by triple junction drag, even in two dimensions with constant and equal surface tensions.
This has implications for numerical implementation of the model via parametrized curves, suggesting perhaps even greater motivation for weak formulations (e.g. in the class of partitions of finite perimeter) and algorithms that represent interfaces implicitly and thus allow seamless handling of topological changes (vs. explicit methods such as front tracking that are likely to become more cumbersome than in the no-drag case).

\item We revisit the study of stable, stationary configurations in the presence of grain rotation and a single junction that was taken up in \cite{ep1}.
By working with widely used (see e.g. \cite{srolo4,holm1,gruber,rohrer_review} among many, many others) orientation and surface tension models such as that of Read \& Shockley \cite{read_shockley}, we obtain {\em different} conclusions regarding the existence and stability of nontrivial stationary states.

\end{enumerate}

\section{The Model}
In this section, we start by recalling the variational model -- the energy, or cost function -- often utilized in mesoscale models of microstructural evolution in polycrystalline materials, going back to the work of Mullins \cite{mullins}.
At the generality of interest to us as well as recent work such as \cite{ep1}, this also requires choosing a {\em surface tension} model that associates to each interface in the network a weight defined in terms of the orientations of the two grains on either side of it.

We will then recall the formal derivation of multiphase curvature motion, with Herring as well as ``triple junction drag'' free boundary conditions, as gradient descent for this energy in suitable metrics.

\subsection{Gradient flow of the total surface energy}
We recall the mesoscale model for grain boundary motion originally introduced by Mullins \cite{mullins}.
Let $\Omega$ be a bounded domain with smooth boundary in $\mathbb{R}^d$, representing the total region occupied by the polycrystalline material.
Grains that constitute the microstructure of the material are modeled as subsets $\Omega_j$ of $\Omega$ that partition it:
\begin{equation}
\label{eq:partition}
\bigcup_{j=1}^N \Omega_j = \Omega \mbox{ and } \sum_{j=1}^N \mbox{Vol}(\Omega_j) = \mbox{Vol}(\Omega)
\end{equation}
where we write $\mbox{Vol}(\Sigma)$ to denote the volume of a set $\Sigma \subset \mathbb{R}^d$.
Hence, morally speaking, two distinct sets $\Omega_i$ and $\Omega_j$, $i\not= j$, can only intersect along their boundaries: $\Omega_i \cap \Omega_j = (\partial\Omega_i) \cap (\partial\Omega_j)$, thus forming a partition of $\Omega$ with no overlaps and leaving no vacuum; see Figure \ref{fig:lattice} for an illustration.
\begin{figure}[h]
\begin{center}
\includegraphics[scale=1.25]{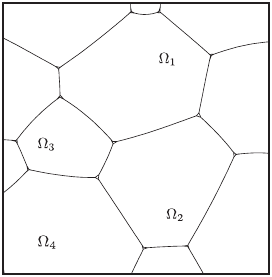}
\caption{\footnotesize Partition of a domain into essentially disjoint phases $\Omega_J$.}
\label{fig:lattice}
\end{center}
\end{figure}
The basic, isotropic version of Mullins' model we study here prescribes an internal energy associated with the microstructure of the material that has the following form:
\begin{equation}
\label{eq:mullinsE}
E(\Omega_1,\Omega_2,\ldots,\Omega_N) = \sum_{i\not= j} \sigma_{i,j}\mbox{Area}\Big( (\partial\Omega_i) \cap (\partial\Omega_j) \Big)
\end{equation}
where we write $\mbox{Area}(\Sigma)$ to denote the $(d-1)$ dimensional surface area of a set $\Sigma$ in $\mathbb{R}^d$.
The positive coefficients $\sigma_{i,j}$ are known as {\em surface tensions} associated with the interfaces $(\partial\Omega_i) \cap (\partial\Omega_j)$; in our setting, they will be determined by the crystallographic orientations $\theta_i$ and $\theta_j$  of the grains $\Omega_i$ and $\Omega_j$ on either side of a given grain boundary $(\partial\Omega_i) \cap (\partial\Omega_j)$.
In Mullins' model, the dynamics of the grain boundaries arise as steepest descent (gradient flow) for energy (\ref{eq:mullinsE}).
Weak formulations of gradient flow  for energy (\ref{eq:mullinsE}) in terms of sets of finite perimeter (the BV-formulation), such as minimizing movements \cite{almgren_taylor_wang,Luckhaus95}, require a {\em triangle inequality} to hold among the $\sigma_{i,j}$ to ensure well-posedness (existence of minimizers):
\begin{equation}
\label{eq:triangle}
\sigma_{i,j} + \sigma_{j,k} \geq \sigma_{i,k} \mbox{ for distinct $i$,$j$, and $k$}.
\end{equation}
Roughly speaking, a physical interpretation for condition (\ref{eq:triangle}) is the preclusion of {\em wetting} (whereby a third grain may form a thin layer coating the interface between two others) and {\em nucleation} (whereby a new grain of nonvanishing size may form spontaneously e.g. in the interior of an existing grain).
Even under this natural condition, during the long time evolution of the network of interfaces that is of interest in applications, many singularities and topological changes may occur that make weak formulations such as the BV-formulation and numerical methods based on them an appropriate choice.

In this study, we instead focus on the free boundary problem that arises as gradient flow for energy (\ref{eq:mullinsE}) starting from smooth, compatible initial data, for short time before any singularities or topological changes to the network may occur.
The novelty is in the choice of the metric, resulting in modified boundary conditions at free boundaries known as triple junctions, modeling what is known as {\em triple junction drag}.
This type of short time existence and uniqueness result is often a very classical first step in developing a more complete theory and, subsequently, computational methods with convergence guarantees.
It appears to be lacking in the literature for our particular boundary conditions, despite recent interest \cite{ep1}.
We closely follow the well known work of Bronsard \& Reitich \cite{bronsardreitich} that first carried out the analogous study in the context of the more familiar Herring angle condition at triple junctions for two dimensional networks of curves.

We thus specialize to two dimensions from here onwards, and in the spirit of \cite{bronsardreitich}, consider the typical scenario in a two dimensional grain network away from singularities, in the vicinity of a triple junction: Just three grains $\Omega_i$, $i\in\{1,2,3\}$ with piecewise smooth curves as their boundaries,  partitioning the domain $\Omega$, which we can take to be a large ball. Note that since we will be fixing the three endpoints $\x_1$, $\x_2$, $\x_3$, the specific choice of the domain will not be relevant to proving existence of solutions to our parametric system.
Thus consider three evolving curves $\Gamma_j(t)\subset\overline{\Omega}$, $j=1,2,3$, each representing the grain boundary between grain $\Omega_{j-1}$ and grain $\Omega_j$.
(For convenience of notation, we will sometimes cyclically refer to grain 3 as grain 0, and to grain 1 as grain 4).
See Figure \ref{fig:network} for an illustration.
\begin{figure}[h]
\includegraphics[scale=0.25]{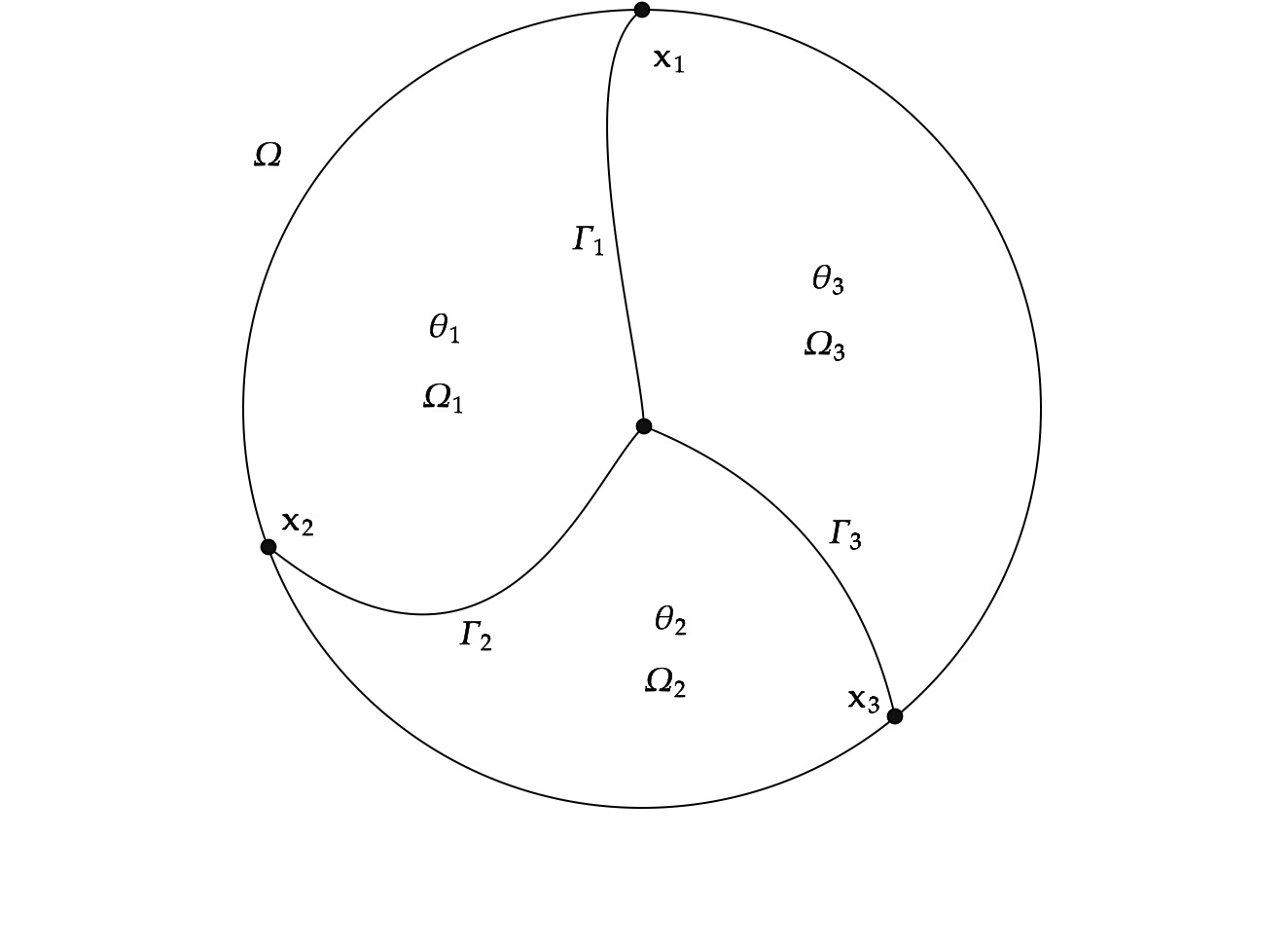}
\centering
\caption{An example of a network}
\label{fig:network}
\end{figure}
For every $t\geq 0$, the curves $\Gamma_j(t)$ will be required to meet at a triple junction at one endpoint, while the other endpoint will be fixed at $\x_j$. 
More precisely, let each curve $\Gamma_j(t)$ be parametrized by $\p_j:[0,1]\times[0,T]\to\overline{\Omega}$, with components given by
\begin{equation}
\begin{split}
    \p_1(x,t) &= (u_1(x,t),u_2(x,t))\\
    \p_2(x,t) &= (u_3(x,t),u_4(x,t))\\
    \p_3(x,t) &= (u_5(x,t),u_6(x,t)).
\end{split}
\end{equation}
We demand that for any $t\geq 0$, the three parametrized curves $( \Gamma_1(t), \Gamma_2(t), \Gamma_3(t) )$ constitute a {\em network} $\mathcal{N}(t)$, which we define as follows:
\begin{definition}[Network]
    A \textit{network} $\mathcal{N}$ refers to the union of the images of three embedded $C^1$ curves $\p_j:[0,1]\to\overline{\Omega}$, $j=1,2,3$. That is, each curve $\p_j(x)$ must be
    \begin{enumerate}[label=(\roman*)]
        \item \textit{regular:} i.e. for any $x\in[0,1]$, we have $\p_{jx}(x)\neq 0$, and
        \item \textit{injective:} i.e. no self-intersections. 
    \end{enumerate}
    Furthermore, we require that the three curves meet \textit{only} at their starting point, i.e.
    \begin{equation}
    \begin{split}
        \bigcap_{j\in\{1,2,3\}} \{ \mathbf{p}_j(x) \, : \, 0 < x \leq 1 \} = \emptyset \mbox{ and } \p_1(0) = \p_2(0) = \p_3(0)
    \end{split}
    \end{equation}
   while the other endpoints are three distinct points on the circle $\partial\Omega$, i.e.
    \begin{equation}
    \begin{split}
       \p_j(1)\equiv \x_j\in\partial\Omega \mbox{ with } \mathbf{x}_i \not= \mathbf{x}_j \mbox{ if } i\not= j.
    \end{split}
    \end{equation}
\end{definition}
\noindent In particular, a network $\mathcal{N}$ partitions $\Omega$ into three connected components $\Omega = \Omega_1\cup\Omega_2\cup\Omega _3$ with $\Omega_j\cap\Omega _k=\emptyset$ if $j\neq k$ (up to a set of Lebesgue measure zero). Each $\Omega_j$ represents the region occupied by grain $j$, with grain orientation $\theta_j$.

Next, we discuss how the surface tension $\sigma_{j,j-1}$ associated with each curve $\Gamma_j$ is assigned.
The choice of surface tensions for a physically reasonable model is a matter of much discussion and research in the materials science literature; see e.g. \cite{rohrer_review}.
In typical simulations with Mullins' model, the grains are assigned crystallographic orientations (elements of $SO(3)$, usually chosen at random); the surface tension associated with the interface between any two neighboring grains is then determined in terms of the misorientation between the two grains on either side of it.
In our simplified two dimensional crystallography setting, the orientation $\theta_j$ for grain $\Omega_j$ is a scalar (an element of $SO(2)$), representing the angle of rotation about a fixed axis required to obtain its crystal lattice from that of a reference configuration.
In physical terms, this corresponds to a fiber textured thin film.
The surface tension $\sigma_{i,j}$ associated with the interface between grains $\Omega_i$ and $\Omega_j$ will then be taken to be
\begin{equation}
    \sigma_{i,j}(t)\coloneqq \sigma(\theta_i(t) - \theta_j(t))
\end{equation}
where $\sigma : \R \to [0,+\infty)$ is a function that satisfies the following properties:
\begin{align}
    \sigma(-\theta) &= \sigma(\theta)  ,  \quad\forall \theta\in\R\label{S1}\tag{S1} \\
    \sigma(2\pi - \theta) &= \sigma(\theta),  \quad\forall \theta\in[0,\pi] \label{S2}\tag{S2}\\
    \sigma(\theta + 2\pi) &= \sigma(\theta),  \quad\forall \theta\in\R\label{S3}\tag{S3}
\end{align}
By \eqref{S1}, $\sigma_{i,j}$ satisfy the symmetry property $\sigma_{i,j} = \sigma_{j,i}$.
We note that the well known surface tension model of Read \& Shockley \cite{read_shockley} fits this description.
For $j\in\{1,2,3\}$, we will let $\theta_j$ depend on time, so that $\theta_j(t)\in\R$ will be the orientation of grain $\Omega_j$ at time $t\geq 0$.
(And according to our cyclical notation, we take $\theta_0(t) \equiv \theta_3(t)$ and $\theta_4(t) \equiv \theta_1(t)$).

For a network $\mathcal{N}$ with grain orientations $\bftheta = (\theta_1, \theta_2, \theta_3)$, the associated total surface energy (\ref{eq:mullinsE}) becomes
\begin{equation}\label{eq:surface-energy}
    E(\mathcal{N},\bftheta) \coloneqq \sum_{j=1}^3 \sigma_{j-1,j}\text{Length}(\Gamma_j) =\sum_{j=1}^3 \sigma_{j-1,j} \int_0^1 |\p_{jx}(x)|\;dx.
\end{equation}
We are interested in gradient flow in a couple of different metrics for this energy with respect to {\em both} the parametrized curves $\Gamma_j$ {\em and} the orientations $\theta_j$ of the grains $\Omega_j$ they delineate.

With respect to the grain orientations $\theta_j$, we adopt the gradient flow dynamics given by
\begin{equation}
\label{eq:rotation}
\begin{split}
    \frac{d\theta_j}{dt} = -\nu\frac{\delta E}{\delta \theta_j} &= \nu\bigg(\sigma'(\theta_{j-1} - \theta_j)\int_0^1 |\p_{jx}(x,t)|\;dx\\
    &- \sigma'(\theta_j - \theta_{j+1})\int_0^1 |\p_{(j+1)x}(x,t)|\;dx\bigg)
\end{split}
\end{equation}
for the sake of simplicity, and to remain close to the choice made in \cite{ep1} so that comparisons can be made in later sections. Here, $\nu\geq 0$ is a mobility factor which we have chosen to be a constant.
However, note that more physically relevant models may entail mobilities that depend on the size of the grains, e.g. as in \cite{esedoglu_rotation}.
We expect such enhancements would be easy to accommodate into much of the discussions of the present study.

To formally derive the evolution of the curves, we need to compute the first variation of the total surface energy functional with respect to the network. Assume that $\bftheta(t)\equiv\bftheta^0$ is constant, i.e. $\sigma_{j-1,j}(t) \equiv \sigma_{j-1,j}^0$. Let $\mathcal{N}^0$ be a network composed of three smooth curves parametrized by $(\p_j^0(x))_{j=1}^3$, $x\in[0,1]$. Let $\bfxi = (\xi_1,\xi_2,\xi_3)$, with 
$\xi_j\in C^\infty([0,1];\R^2)$ satisfying 
\begin{equation}
\begin{split}
    \xi_1(0) &= \xi_2(0) = \xi_3(0)\\
    \xi_j (1) &= 0, \qquad j=1,2,3.
\end{split}
\end{equation}
For each small $t>0$, let $\mathcal{N}(t)$ denote the network composed of the curves $\p_j^0(x) + t\xi_j(x)$. Following a standard computation of the first variation of Length$(\cdot)$, we get
\begin{equation}
\begin{split}
    \frac{d}{dt}\bigg|_{t=0}E(\mathcal{N}(t),\bftheta^0) &= \frac{d}{dt}\bigg|_{t=0} \sum_{j=1}^3 \sigma_{j-1,j}^0 \int_0^1 |\p_{jx}(x,t)|\;dx\\
    &= \sum_{j=1}^3 \int_0^1 -\sigma_{j-1,j}^0\bfkappa_j^0(x)\cdot\xi^j(x) |\p_{jx}^0(x)|\;dx - \bigg(\sum_{j=1}^3 \sigma_{j-1,j}^0\frac{\p_{jx}^0(0)}{|\p_{jx}^0(0)|}\bigg)\cdot \xi^1(0),
\end{split}
\end{equation}
where $\bfkappa_j^0 = \kappa_j^0 \bfnu_j$ is the curvature vector of the curve $\p_j^0$.
If we let $\mu>0$ represent an inverse mobility for the triple junction $\p_j^0(0)$, then we can write
\begin{equation}\label{firstvar}
    \frac{d}{dt}\bigg|_{t=0}E(\mathcal{N}(t),\bftheta^0)  =\sum_{j=1}^3 \int_0^1 -\sigma_{j-1,j}^0\bfkappa_j^0(x)\cdot\xi^j(x) |\p_{jx}^0(x)|\;dx - \mu\bigg(\sum_{j=1}^3 \sigma_{j-1,j}^0 \frac{1}{\mu}\frac{\p_{jx}^0(0)}{|\p_{jx}^0(0)|}\bigg)\cdot \xi^1(0).
\end{equation}

\begin{definition}\label{def:innerproduct}
    Let $\mathcal{N}^0$ be a network comprising three smooth curves parametrized by $(\p_j^0(x))_{j=1}^3$, $x\in[0,1]$. Suppose that the three curves meet only at a triple junction $\p_1^0(0) = \p_2^0(0) = \p_3^0(0)$. Let $\bfxi = (\xi_1,\xi_2,\xi_3)$, $\bfeta = (\eta_1,\eta_2,\eta_3)$ be vector fields along $\mathcal{N}_0$ with 
    $\xi_j,\;\eta_j\in C^\infty([0,1];\R^2)$ satisfying 
    \begin{equation}
    \begin{split}
        \xi_1(0) &= \xi_2(0) = \xi_3(0)\\
        \eta_1(0) &= \eta_2(0) = \eta_3(0)\\
        \xi_j (1) &= \eta_j(1) = 0, \qquad j=1,2,3.
    \end{split}
    \end{equation}
     For $\mu \geq 0$, we define the inner product involving the normal velocities along the curves and the pointwise velocities at the triple junction (modulo degeneracy in the tangential direction)
    \begin{equation}
        \langle \bfxi , \bfeta  \rangle_{\mu,\mathcal{N}^0} \coloneqq \sum_{j=1}^3 \int_0^1 \big(\xi^j(x) \cdot \bfnu_j(x)\big)\big( \eta^j(x)\cdot \bfnu_j(x)\big) |\p_{jx}^0(x)|\;dx + \mu \xi^1 (0)\cdot \eta^1(0).
    \end{equation}
\end{definition}

If $\mu>0$, this inner product $\langle,\rangle_{\mu,\mathcal{N}}$ contains both a 1-dimensional component (along the curves) and a 0-dimensional component (at the triple junction). If $\mu = 0$, this inner product reduces to the $L^2$ inner product of normal velocities along the network, which induces the Herring angle condition.

Using this inner product with $\mu>0$, \eqref{firstvar} reads
\begin{equation}
    \frac{d}{dt}\bigg|_{t=0}E(\mathcal{N}(t),\bftheta^0) = \bigg\langle \frac{\delta E}{\delta \mathcal{N}}(\mathcal{N}^0, \bftheta^0), \bfxi  \bigg\rangle_{\mu,\mathcal{N}^0}.
\end{equation}

 Thus, (formally) the gradient flow of the total surface energy functional with respect to this new metric is described by the following system of PDEs with a dynamic boundary condition at the triple junction:
\begin{align}
    \p_{jt}^\perp \coloneqq  \big( \p_{jt}\cdot\bfnu_j \big) \bfnu_j &= \sigma_{j-1,j}\bfkappa_j \label{normal}\\
    \p_{1t}(0,t) = \p_{2t}(0,t) &= \p_{3t}(0,t) = \frac{1}{\mu}\sum_{j=1}^3 \sigma_{j-1,j} \frac{\p_{jx}(0,t)}{|\p_{jx}(0,t)|}\label{drag}.
\end{align}

\subsection{The parametric problem}

We will consider the evolution of parametrized curves by the vectorial equation
\begin{equation}
\label{eq:specialflow}
     \p_{jt} = \sigma_{j-1,j}(t)\frac{\p_{jxx}}{|\p_{jx}|^2}
\end{equation}
which is commonly referred to as a \textit{special flow} \cite{Mantegazza2019}.
For the sake of completeness, we recall the derivation, which entails a nontrivial choice of tangential velocity.
Indeed, the curvature vector associated to the interface with parametrization $\mathbf{p}_j$ is given by
\begin{equation}
\label{eq:kvector}
\frac{1}{|\mathbf{p}_j|} \partial_x \left( \frac{\mathbf{p}_{jx}}{|\mathbf{p}_{jx}|} \right) = \frac{\mathbf{p}_{jxx}}{|\mathbf{p}_{jx}|^2} - \frac{\mathbf{p}_{jx} \big( \mathbf{p}_{jx} \cdot \mathbf{p}_{jxx} \big)}{|\mathbf{p}_{jx}|^2}
\end{equation}
Flow (\ref{eq:specialflow}) is obtained as a description of curvature motion by noting that the second term in the right hand side of (\ref{eq:kvector}) is in the tangent direction; hence the normal component of the velocity $\mathbf{p}_{jt}$ in (\ref{eq:specialflow}) matches precisely the curvature vector (\ref{eq:kvector}).
More precisely, we introduce the following parametric problem:

\begin{definition}[Parametric problem]\mbox{}
    We say that $((\p_j(x,t))_{j=1}^3 , (\theta_j(t))_{j=1}^3)$ is a solution to the \emph{parametric problem} starting at initial data $((\p_j^0(x))_{j=1}^3 , (\theta_j^0)_{j=1}^3)$ if it satisfies the following system
    
\begin{equation}
\begin{aligned}
    \p_{jt} &= \sigma_{j-1,j}(t)\frac{\p_{jxx}}{|\p_{jx}|^2}, \quad &(x,t)\in[0,1]\times[0,T], \quad j=1,2,3\\
    \p_j(x,0) &= \p_j^0(x), \quad &x\in[0,1],\quad j=1,2,3\\
    \p_1(0,t) &= \p_2(0,t) = \p_3(0,t), &\text{at }x=0, \quad t\in[0,T]\\
    \p_{1t}(0,t) &= \frac{1}{\mu}\sum_{j=1}^3 \sigma_{j-1,j}(t)\frac{\p_{jx}(0,t)}{|\p_{jx}(0,t)|},&\text{at }x=0, \quad t\in[0,T]\\
    \p_j(1,t) &\equiv \x_j, \quad &t\in[0,T] ,\quad j=1,2,3\\
    \frac{d\theta_j}{dt} = -\nu\frac{\delta E}{\delta \theta_j} &= \nu\bigg(\sigma'(\theta_{j-1} - \theta_j)\int_0^1 |\p_{jx}(x,t)|\;dx\\
    &-\sigma'(\theta_j - \theta_{j+1})\int_0^1 |\p_{(j+1)x}(x,t)|\;dx\bigg),  \qquad & t\in[0,T],\quad j=1,2,3\\
    \theta_j(0) &= \theta_j^0,   &j=1,2,3
\end{aligned}
\end{equation}

\end{definition}

\newpage

We rewrite this system component-wise, with $\p_1 = (u_1,u_2)$, $\p_2 = (u_3,u_4)$, $\p_3 = (u_5,u_6)$: 

\begin{equation}\label{pde}
\begin{aligned}
    u_{jt} &= \frac{\sigma_{3,1}}{|\p_{1x}|^2}u_{jxx}, &(x,t)\in[0,1]\times[0,T], \quad j=1,2 \\
    u_{jt} &= \frac{\sigma_{1,2}}{|\p_{2x}|^2}u_{jxx}, &(x,t)\in[0,1]\times[0,T], \quad j=3,4 \\
    u_{jt} &= \frac{\sigma_{2,3}}{|\p_{3x}|^2}u_{jxx}, &(x,t)\in[0,1]\times[0,T], \quad j=5,6 \\
    u_j(x,0) &= u_j^0(x), &x\in[0,1], \quad j=1,\ldots,6 \\
    u_j(0,t) &= u_{j+2}(0,t) = u_{j+4}(0,t), &t\in[0,T], \quad j=1,2 \\
    u_{1t} &= \frac{1}{\mu}\bigg( \frac{\sigma_{3,1}}{|\p_{1x}|}u_{1x} + \frac{\sigma_{1,2}}{|\p_{2x}|}u_{3x} + \frac{\sigma_{2,3}}{|\p_{3x}|}u_{5x}\bigg), &\text{ at }x=0, \quad t\in[0,T] \\
    u_{2t} &= \frac{1}{\mu}\bigg( \frac{\sigma_{3,1}}{|\p_{1x}|}u_{2x} + \frac{\sigma_{1,2}}{|\p_{2x}|}u_{4x} + \frac{\sigma_{2,3}}{|\p_{3x}|}u_{6x}\bigg), &\text{ at }x=0, \quad t\in[0,T] \\
    \big( u_1(1,t) , u_2(1,t) \big) &\equiv \mathbf{x}_1, & t\in[0,T] \\
    \big( u_3(1,t) , u_4(1,t) \big) &\equiv \mathbf{x}_2, & t\in[0,T] \\
    \big( u_5(1,t) , u_6(1,t) \big) &\equiv \mathbf{x}_3, & t\in[0,T] \\
    \frac{d\theta_j}{dt} = -\nu\frac{\delta E}{\delta \theta_j} &= \nu\bigg(\sigma'(\theta_{j-1} - \theta_j)\int_0^1 |\p_{jx}(x,t)|\;dx \\
    &- \sigma'(\theta_j - \theta_{j+1})\int_0^1 |\p_{(j+1)x}(x,t)|\;dx\bigg), &t\in[0,T],\quad j=1,2,3 \\
    \theta_j(0) &= \theta_j^0, &j=1,2,3
\end{aligned}
\end{equation}

\section{Main Results}
Before stating our short time existence theorem, we discuss {\em compatibility conditions} that need to be placed on initial data.

\begin{definition}[Parametric compatibility conditions]
\label{def:parcomp}
    We say that the initial data $((u_j^0(x))_{j=1}^6 , (\theta_j^0)_{j=1}^3)$ satisfies the compatibility conditions for \eqref{pde} if
    \begin{align}
        u_1^0 = u_3^0 &= u_5^0  \nonumber\\
        u_2^0 = u_4^0 &= u_6^0  \nonumber\\
        \frac{\sigma_{3,1}^0}{|\p_{1x}^0|^2}u_{1xx}^0 = \frac{\sigma_{1,2}^0 }{|\p_{2x}^0|^2}u_{3xx}^0 = \frac{\sigma_{2,3}^0 }{|\p_{3x}^0|^2}u_{5xx}^0 &= \frac{1}{\mu}\bigg(  \frac{\sigma_{3,1}^0}{|\p_{1x}^0|} u_{1x}^0 + \frac{\sigma_{1,2}^0}{|\p_{2x}^0|} u_{3x}^0 + \frac{\sigma_{2,3}^0}{|\p_{3x}^0|} u_{5x}^0  \bigg) \label{compatibility1}\\
        \frac{\sigma_{3,1}^0}{|\p_{1x}^0|^2}u_{2xx}^0 = \frac{\sigma_{1,2}^0 }{|\p_{2x}^0|^2}u_{4xx}^0 = \frac{\sigma_{2,3}^0 }{|\p_{3x}^0|^2}u_{6xx}^0 &= \frac{1}{\mu}\bigg(  \frac{\sigma_{3,1}^0}{|\p_{1x}^0|} u_{2x}^0 + \frac{\sigma_{1,2}^0}{|\p_{2x}^0|} u_{4x}^0 + \frac{\sigma_{2,3}^0}{|\p_{3x}^0|} u_{6x}^0  \bigg) \label{compatibility2}
    \end{align}
    at $x=0$ and
\begin{equation}\label{compatibility3}
\begin{aligned}
    \p_j^0 &= \x_j, \qquad &j=1,2,3 \\
    \p_{jxx}^0 &= \mathbf{0}, \qquad & j=1,2,3 
\end{aligned}
\end{equation}
at $x=1$.
\end{definition}

\noindent Before stating our main result, we first introduce helpful notation.
Let $Q_T \coloneqq [0,1]\times[0,T]$.
We will use the parabolic H\"{o}lder spaces (see \cite{ladyzhenskaia} Ch.1 p.8) $C^{2+\alpha , 1+\alpha/2}(Q_T)$  with the norm $|\cdot|_{Q_T}^{(2+\alpha , 1+\alpha/2)}$. For a function $u(x,t)$, this norm is defined by 
\begin{equation}
\begin{split}
    |u|_{Q_T}^{(2+\alpha , 1+ \alpha/2)} &\coloneqq |u|_{Q_T}^{(0)} + |u_x|_{Q_T}^{(0)} + |u_{xx}|_{Q_T}^{(0)} + |u_t|_{Q_T}^{(0)} \\
    &+ \langle u_{xx} \rangle_{x,Q_T}^{(\alpha)} + \langle u_t \rangle_{x,Q_T}^{(\alpha)}  + \langle u_x \rangle_{t,Q_T}^{((1+\alpha)/2)} + \langle u_{xx} \rangle_{t,Q_T}^{(\alpha/2)} + \langle u_t \rangle_{t,Q_T}^{(\alpha/2)}  
\end{split}
\end{equation}
where
\begin{equation}
\begin{split}
    |u|_{Q_T}^{(0)} &\coloneqq \max_{(x,t)\in Q_T}|u(x,t)|\\
     \langle u \rangle_{x,Q_T}^{(\alpha)} & \coloneqq \sup_{(x,t),(y,t)\in Q_T,\; x\neq y} \frac{|u(x,t) - u(y,t)|}{|x-y|^\alpha}\\
     \langle u \rangle_{t,Q_T}^{(\alpha)} & \coloneqq \sup_{(x,s),(x,t)\in Q_T, \; s\neq t} \frac{|u(x,s) - u(x,t)|}{|s-t|^\alpha}.
\end{split}
\end{equation}
For a function $f(z)$, $z\in\mathcal{D}$ of a single variable, we will use the notation
\begin{equation}
\begin{split}
    \langle f\rangle^{(\alpha)} &\coloneqq \sup_{z,w\in\mathcal{D}. \; z\neq w}\frac{|f(z) - f(w)|}{|z-w|^\alpha}\\
    |f|^{(k+\alpha)} & \coloneqq \sum_{j=0}^k |f^{(j)}|^{(0)} + \langle f^{(k)}\rangle^{(\alpha)} = \sum_{j=0}^k \max_{z\in\mathcal{D}}|f^{(j)}(z)| + \langle f^{(k)}\rangle^{(\alpha)}
\end{split}
\end{equation}
to denote the usual H\"{o}lder seminorm and H\"{o}lder norm respectively. We will often omit the subscript indicating the domain whenever the intention is clear.

\begin{theorem}[Parametric well-posedness]\label{maintheorem}
    Let $\alpha\in (0,1)$, $u_j^0\in C^{2+\alpha}([0,1])$, $j=1,\ldots,6$. For $j=1,2,3$, let $\theta_j^0 \in [0,2\pi)$ be the initial orientation of grain $j$. 
    Suppose that the initial data $((u_j^0(x))_{j=1}^6 , (\theta_j^0)_{j=1}^3)$ satisfies the compatibility conditions for \eqref{pde}. Assume that
    \begin{equation}
        \delta \coloneqq \min_{j=1,2,3} \inf_{x\in[0,1]} |\p_{jx}^0 (x)|>0.
    \end{equation}
    Let $\sigma\in C^{1}(\R\setminus 2\pi\mathbb{Z} \;; [0,+\infty)) \cap C^0(\R \; ; [0,+\infty))$ be a non-negative function satisfying \eqref{S1}, \eqref{S2} and \eqref{S3}. Further assume that $\sigma$ and $\sigma'$ are Lipschitz with Lipschitz constants $Lip(\sigma), \; Lip(\sigma')$ respectively. Suppose that
    \begin{equation}
        \min_{j=1,2,3} \sigma_{j-1,j}^0 = \min_{j=1,2,3} \sigma(\theta_{j-1}^0 - \theta_j^0) > 0.
    \end{equation}
    Then there exists $M = M(|u_j^0|^{(2+\alpha)},|x_j|,\theta_j^0 , \mu , \sigma)>0$ and 
    $T = T(|u_j^0|^{(2+\alpha)},|x_j|,\theta_j^0,\delta,\mu,\nu,\sigma) > 0$ such that the system \eqref{pde} admits a unique solution
    \begin{equation}
        ((u_j(x,t))_{j=1}^6 , (\theta_j(t))_{j=1}^3) \in C^{2+\alpha, 1+\alpha/2}([0,1]\times[0,T];\R^6) \times C^1([0,T];\R^3), 
    \end{equation}
    with
    \begin{equation}
        |u_j|_{Q_T}^{(2+\alpha, 1+\alpha/2)},\; |\theta_k|^{(\alpha)} \leq M, \qquad j=1,\ldots, 6, \; k=1,2,3.
    \end{equation}
\end{theorem}

Conditions given in Definition \ref{def:parcomp}
are the zeroth and first order compatibility conditions in \cite{ladyzhenskaia}. Note that the first order conditions \eqref{compatibility1} and \eqref{compatibility2} depend on the choice of parametrization $\p_j(\cdot,t)$. This means that in order to determine if an initial (unparametrized) network is compatible with the parametric problem, one must find a particular parametrization of this network that satisfies the parametric compatibility conditions. It is not immediately clear whether such a parametrization exists. Thus, it is natural to look for a set of geometric compatibility conditions (see Definition \ref{def:geocomp} below) on the initial curves (and initial grain orientations) that will guarantee the existence of a parametrization satisfying the parametric compatibility conditions.

It is also reasonable to ask if these parametric compatibility conditions can be relaxed by considering a slightly weaker notion of solutions. The elimination of conditions on derivatives of initial curves at the triple junction (and also the fixed endpoints) would allow us to accommodate a wider class of initial curves. In Section \ref{sec:sobolev}, we prove the existence of solutions in a wider class of curves using a similar strategy that was used to prove Theorem \ref{maintheorem}.

\begin{definition}[Geometric compatibility conditions]\label{def:geocomp}
    Given initial grain orientations $\theta_j^0$, we say that an initial network $\mathcal{N}^0$, composed of curves $\Gamma_j^0\subset\overline{\Omega}$, satisfies the \textit{geometric compatibility conditions} if the curvatures satisfy
    \begin{equation}\label{geocompatibility1}
         \sigma_{j-1,j}^0\kappa_j^0 = \frac{1}{\mu}\bigg( \sum_{k=1}^3 \sigma_{k-1,k}^0 \bftau_k^0 \bigg)\cdot\bfnu_j^0
    \end{equation}
    at the triple junction,
    \begin{equation}\label{geocompatibility2}
        \kappa_j^0 = 0
    \end{equation}
    at the endpoints $\x_j$, and each curve $\Gamma_j^0$ can be parametrized by $\q_j^0 \in C^{2+\alpha}([0,1];\overline{\Omega})$ with $|\q_{jx}^0(x)|\neq 0$ for all $x\in[0,1]$. Here, $\bftau_k^0$ is the unit tangent of $\Gamma_k^0$ at the triple junction.
\end{definition}

\begin{lemma}[Equivalence of parametric and geometric compatibility     conditions]\label{equivalencecompatibility}
    A $C^{2+\alpha}$ initial network $\mathcal{N}^0$, composed of curves $\Gamma_j^0\subset\overline{\Omega}$ satisfies the geometric compatibility conditions if and only if there exist parametrizations $\p_j^0 \in C^{2+\alpha}([0,1];\overline{\Omega})$ with $|\p_{jx}^0(x)|\neq 0$ for all $x\in[0,1]$ and satisfying the parametric compatibility conditions.
\end{lemma}

\begin{proof}
Such an equivalence between parametric and geometric compatibility has been established in \cite{Mantegazza2019} for the Herring angle condition. We follow a similar approach for triple junction drag. Suppose the curves $\Gamma_j^0$ can be parametrized by $\p_j^0 \in C^{2+\alpha}([0,1];\overline{\Omega})$ with $|\p_{jx}^0(x)|\neq 0$ for all $x\in[0,1]$ and satisfying the parametric compatibility conditions. Then \eqref{compatibility1} and \eqref{compatibility2} imply that
\begin{equation}
\begin{split}
    \sigma_{j-1,j}^0 \frac{\p_{jxx}^0}{|\p_{jx}^0|^2} = \frac{1}{\mu}\sum_{k=1}^3 \sigma_{k-1,k}^0 \frac{\p_{kx}^0}{|\p_{kx}^0|} = \frac{1}{\mu}\sum_{k=1}^3 \sigma_{k-1,k}^0 \bftau_k^0, \qquad j=1,2,3\\
    \implies \sigma_{j-1,j}^0\kappa_j^0 = \sigma_{j-1,j}^0 \frac{\p_{jxx}^0}{|\p_{jx}^0|^2} \cdot \bfnu_j^0 = \frac{1}{\mu}\bigg(\sum_{k=1}^3 \sigma_{k-1,k}^0 \bftau_k^0\bigg)\cdot\bfnu_j^0 , \qquad j=1,2,3
\end{split}
\end{equation}
at the triple junction. Meanwhile, \eqref{compatibility3} implies that
\begin{equation}
\begin{split}
    \kappa_j^0 = \frac{\p_{jxx}^0}{|\p_{jx}^0|^2}\cdot\bfnu_j^0 = 0
\end{split}
\end{equation}
at the endpoint $\x_j$.

Conversely, suppose that the curves $\Gamma_j^0$, parametrized by $\q_j^0(x)$, $x\in[0,1]$ satisfy the geometric compatibility conditions. We will construct reparametrizations $\p_j^0(x)\coloneqq \q_j^0(\varphi_j(x))$ satisfying the parametric compatibility conditions.

By Lemma \ref{lem:diffeo}, there exists diffeomorphisms $\varphi_j:[0,1]\to[0,1]$ satisfying
\begin{equation}
\begin{split}
    \varphi_j(0) &= 0, \\
    \varphi_j(1) &= 1, \\
    \varphi_{jx}(0) = \varphi_{jx}(1) &= 1,\\
    \varphi_{jx}(x) &> 0, \qquad \forall x\in[0,1],\\
    \varphi_{jxx}(0) &= |\q_{jx}^0|\bigg[ \frac{1}{\mu \sigma_{j-1,j}^0}\bigg( \sum_{k=1}^3 \sigma_{k,k-1}^0\frac{\q_{kx}^0}{|\q_{kx}^0|} \bigg)\cdot\frac{\q_{jx}^0}{|\q_{jx}^0|} - \frac{\q_{jxx}^0\cdot\q_{jx}^0}{|\q_{jx}^0|^3} \bigg]\bigg|_{x=0},\\
    \varphi_{jxx}(1) &= - \frac{\q_{jxx}^0\cdot\q_{jx}^0}{|\q_{jx}^0|^2} \bigg|_{x=1}.
\end{split}
\end{equation} 
One can check that this choice of reparametrization satisfies
\begin{equation}\label{lamda}
    \sigma_{j-1,j}^0\frac{\p_{jxx}^0\cdot\p_{jx}^0}{|\p_{jx}^0|^3} = \frac{1}{\mu}\bigg( \sum_{k=1}^3 \sigma_{k,k-1}^0\bftau_k^0 \bigg)\cdot\bftau_j^0
\end{equation}
at $x=0$ (i.e. triple junction) and
\begin{equation}
    \frac{\p_{jxx}^0\cdot\p_{jx}^0}{|\p_{jx}^0|^3} = 0
\end{equation}
at $x=1$ (i.e. endpoints). Thus, at $x=1$, we have
\begin{equation}
\begin{split}
    \frac{\p_{jxx}^0}{|\p_{jx}^0|^2} &= \bigg(\frac{\p_{jxx}^0}{|\p_{jx}^0|^2}\cdot \bfnu_j^0\bigg)\bfnu_j^0 + \bigg(\frac{\p_{jxx}^0}{|\p_{jx}^0|^2}\cdot \bftau_j^0\bigg)\bftau_j^0\\
    &= \kappa_j^0 \bfnu_j^0 + \bigg(\frac{\p_{jxx}^0\cdot\p_{jx}^0}{|\p_{jx}^0|^3}\bigg)\bftau_j^0\\
    &= 0,
\end{split}
\end{equation}
satisfying \eqref{compatibility3}. At $x=0$, we have
\begin{equation}
\begin{split}
    \sigma_{j-1,j}^0\frac{\p_{jxx}^0}{|\p_{jx}^0|^2} &= \sigma_{j-1,j}^0\bigg(\frac{\p_{jxx}^0}{|\p_{jx}^0|^2}\cdot \bfnu_j^0\bigg)\bfnu_j^0 + \sigma_{j-1,j}^0\bigg(\frac{\p_{jxx}^0}{|\p_{jx}^0|^2}\cdot \bftau_j^0\bigg)\bftau_j^0\\
    &= \bigg(\frac{1}{\mu}\bigg( \sum_{k=1}^3 \sigma_{k-1,k}^0 \bftau_k^0 \bigg)\cdot\bfnu_j^0\bigg) \bfnu_j^0 + \bigg(\frac{1}{\mu}\bigg( \sum_{k=1}^3 \sigma_{k,k-1}^0\bftau_k^0 \bigg)\cdot\bftau_j^0\bigg)\bftau_j^0, \qquad \text { by }\eqref{geocompatibility1},\eqref{lamda}\\
    &= \frac{1}{\mu}\sum_{k=1}^3 \sigma_{k-1,k}^0 \bftau_k^0 ,
\end{split}
\end{equation}
satisfying \eqref{compatibility1}, \eqref{compatibility2}.
\end{proof}

Notice that \eqref{normal},\eqref{drag} is a purely geometric evolution problem that can be stated for a network of curves (viewed as subsets of $\R^2$) without reference to a specific parametrization. While a solution of the parametric problem described by system \eqref{pde} clearly solves the geometric problem, it also prescribes a specific parametrization of the curves. It is therefore natural to ask whether the geometric problem \eqref{normal},\eqref{drag} has a unique solution. We answer this question positively, adapting the strategy from \cite{Mantegazza2019,mantegazzanovagatortorelli} which was used to answer a similar question in the case of Herring angle conditions.

\begin{proposition}[Geometric uniqueness]
    Let $\mathcal{N}^0$ be a $C^{2+\alpha}$ initial network composed of curves $\Gamma_j^0 \subset\overline{\Omega}$ that satisfy the geometric compatibility conditions. Assume that the surface tensions $\sigma_{i,j} = 1$ for all $i,j$ and that $\mu>0$. Then $\mathcal{N}^0$ generates a unique solution to the geometric problem \eqref{normal},\eqref{drag} on some time interval $[0,\Tilde{T}]$.
\end{proposition}

\begin{proof}
    Let $(\p_j(x,t))_{j=1}^3$ be a special flow (i.e., a solution to the parametric problem \eqref{pde}) starting from $\mathcal{N}^0$, defined on some time interval $[0,T]$. The existence of this special flow is guaranteed by Theorem \ref{maintheorem} and Lemma \ref{equivalencecompatibility}. In particular, the geometric problem \eqref{normal}, \eqref{drag} has at least one solution.

    Suppose there exists another solution to the geometric problem on some time interval $[0,T']$, which we shall parametrize by $(\q_j(x,t))_{j=1}^3$. We may assume that $\q_j(\cdot,0) = \p_j(\cdot,0)$. We want to show that the curves $\p_j(\cdot,t)$ and $\q_j(\cdot,t)$ coincide as subsets of $\R^2$ on some time interval $[0,\Tilde{T}]$ by proving the existence of a time-dependent reparametrization $(\varphi_j(x,t))_{j=1}^3$ with $(x,t)\in[0,1]\times[0,\Tilde{T}]$ so that $\p_j(x,t) = \q_j(\varphi_j(x,t),t)$. To do so, we derive a system of quasilinear parabolic PDEs that the $\varphi_j$ need to satisfy.

    Assuming that $\varphi_{jx}>0$, $\varphi_j(x,0) = x$, $\varphi_j(0,t) = 0$, $\varphi_j(1,t) = 1$, we set $\Tilde{\q}_j(x,t) = \q(\varphi_j(x,t),t)$ to see that
    \begin{equation}
    \begin{split}
        \Tilde{\q}_{jt}(x,t) &= \q_{jt}(\varphi_j,t) + \q_{jx}(\varphi_j,t)\varphi_{jt}\\
        &= (\q_{jt}(\varphi_j,t) \cdot  \bfnu_j(\varphi_j,t))\bfnu_j(\varphi_j,t) + (\q_{jt}(\varphi_j,t) \cdot \bftau_j(\varphi_j,t))\bftau_j(\varphi_j,t) + \q_{jx}(\varphi_j,t)\varphi_{jt}\\
        &= \kappa_j(\varphi_j,t)\bfnu_j(\varphi_j,t)) + \lambda_j(\varphi_j,t)\bftau_j(\varphi_j,t) + \q_{jx}(\varphi_j,t)\varphi_{jt}\\
        &= \dfrac{\q_{jxx}(\varphi_j,t) \cdot \bfnu_j (\varphi_j,t)}{|\q_{jx}(\varphi_j,t)|^2}\bfnu_j(\varphi_j,t) + \lambda_j(\varphi_j,t)\bftau_j(\varphi_j,t) + \q_{jx}(\varphi_j,t)\varphi_{jt}
    \end{split}
    \end{equation}
    where $\lambda_j(\varphi_j,t) \coloneqq \q_{jt}(\varphi_j,t) \cdot \bftau_j(\varphi_j,t)$ is the tangential velocity at $\q_j(\varphi_j,t)$.

    On the other hand, we have
    \begin{equation}
        \dfrac{\Tilde{\q}_{jxx}(x,t)}{|\Tilde{\q}_{jx}(x,t)|^2} = \dfrac{\q_{jxx}(\varphi_j,t)(\varphi_{jx})^2 + \q_{jx}(\varphi_j,t)\varphi_{jxx}}{|\q_{jx}(\varphi_j,t)\varphi_{jx}|^2}.
    \end{equation}

    Thus, we demand that $\varphi_j(x,t)$ satisfies
    \begin{equation}
    \label{param_pde1}
    \begin{split}
        &\dfrac{\q_{jxx}(\varphi_j,t) \cdot \bfnu_j (\varphi_j,t)}{|\q_{jx}(\varphi_j,t)|^2}\bfnu_j(\varphi_j,t) + \lambda_j(\varphi_j,t)\bftau_j(\varphi_j,t) + \q_{jx}(\varphi_j,t)\varphi_{jt} \\
        &= \dfrac{\q_{jxx}(\varphi_j,t)(\varphi_{jx})^2 + \q_{jx}(\varphi_j,t)\varphi_{jxx}}{|\q_{jx}(\varphi_j,t)\varphi_{jx}|^2}
    \end{split}
    \end{equation}
    so that $\Tilde{\q}_j (x,t)$ satisfies
    \begin{equation}
        \Tilde{\q}_{jt}(x,t) = \dfrac{\Tilde{\q}_{jxx}(x,t)}{|\Tilde{\q}_{jx}(x,t)|^2}.
    \end{equation}
    Note that $\Tilde{\q}_j(x,0) = \q_j(\varphi_j(x,0),0) = \q_j(x,0) = \p_j(x,0)$. By the uniqueness statement of Theorem \ref{maintheorem}, we conclude that $\Tilde{\q}_j$ coincides with $\p_j$ on their common time interval of existence, which proves the Proposition.

    It remains to justify the existence of $\varphi_j(x,t)$ satisfying \eqref{param_pde1}. To rewrite \eqref{param_pde1} as a system of quasilinear parabolic PDEs, we take the dot product of \eqref{param_pde1} with $\bfnu_j(\varphi_j,t)$ and $\bftau_j(\varphi_j,t)$ respectively to see that $\varphi_j$ should satisfy 
    \begin{equation}
        \varphi_{jt} - \dfrac{\varphi_{jxx}}{|\q_{jx}(\varphi_j,t)\varphi_{jx}|^2} = \dfrac{\q_{jxx}(\varphi_j,t)}{|\q_{jx}(\varphi_j,t)|^2}\cdot \dfrac{\q_{jx}(\varphi_j,t)}{|\q_{jx}(\varphi_j,t)|^2} - \dfrac{\lambda_j(\varphi_j,t)}{|\q_{jx}(\varphi_j,t)|}
    \end{equation}
    subject to the initial and boundary conditions
    \begin{equation}
        \varphi_j(x,0) = x, \qquad \varphi_j(0,t) = 0, \qquad \varphi_j(1,t) = 1.
    \end{equation}
    This is the same system as (\cite{Mantegazza2019} Theorem 3.11). Thus, it suffices to check that the initial data $\varphi_j(x,0) = x$ satisfies the compatibility conditions at $x=0$ and $x=1$, namely,
    \begin{equation}
        \dfrac{\varphi_{jxx}}{|\q_{jx}(\varphi_j,t)\varphi_{jx}|^2} + \dfrac{\q_{jxx}(\varphi_j,t)}{|\q_{jx}(\varphi_j,t)|^2}\cdot \dfrac{\q_{jx}(\varphi_j,t)}{|\q_{jx}(\varphi_j,t)|^2} - \dfrac{\lambda_j(\varphi_j,t)}{|\q_{jx}(\varphi_j,t)|} \bigg|_{t=0, x\in\{0,1\}} = 0. 
    \end{equation}

    Note that $\varphi_{jxx}(x,0) = 0$ and so at $(x,t)=(0,0)$, we have
    \begin{equation}
    \begin{split}
        \dfrac{\q_{jxx}(0,0)}{|\q_{jx}(0,0)|^2}\cdot \dfrac{\q_{jx}(0,0)}{|\q_{jx}(0,0)|^2} - \dfrac{\lambda_j(0,0)}{|\q_{jx}(0,0)|} &= \dfrac{\q_{jxx}(0,0)}{|\q_{jx}(0,0)|^2}\cdot \dfrac{\q_{jx}(0,0)}{|\q_{jx}(0,0)|^2} - \dfrac{\q_{jt}(0,0) \cdot \bftau_j(0,0)}{|\q_{jx}(0,0)|}\\
        &= \dfrac{1}{|\q_{jx}(0,0)|}\bigg(\dfrac{\q_{jxx}(0,0)}{|\q_{jx}(0,0)|^2} - \q_{jt}(0,0)\bigg)\cdot \bftau_j(0,0)\\
        &= \dfrac{1}{|\q_{jx}(0,0)|}\cancelto{0}{\bigg(\dfrac{\q_{jxx}(0,0)}{|\q_{jx}(0,0)|^2} - \frac{1}{\mu}\sum_{i=1}^3 \bftau_j(0,0)\bigg)}\cdot \bftau_j(0,0)\\
        &= 0
    \end{split}
    \end{equation}
    because $\q_j(x,0) = \p_j(x,0)$ satisfies the parametric compatibility conditions in Definition \ref{def:parcomp}. At $(x,t) = (1,0)$, we use the fact that $\q_{jxx}(1,0) = 0$ and $\q_{jt}(1,t)=0$ to deduce that
    \begin{equation}
        \dfrac{\q_{jxx}(1,0)}{|\q_{jx}(1,0)|^2}\cdot \dfrac{\q_{jx}(1,0)}{|\q_{jx}(1,0)|^2} - \dfrac{\lambda_j(1,0)}{|\q_{jx}(1,0)|} = 0.
    \end{equation}

\end{proof}

\section{Proof of Theorem \ref{maintheorem}}
To prove Theorem \ref{maintheorem}, we will follow the approach of Bronsard \& Reitich in \cite{bronsardreitich}, which involves a linearization of \eqref{pde} about the initial data and a fixed-point argument. 

Let 
\begin{equation}
\begin{aligned}
    X_j &\coloneqq \{ u\in C^{2+\alpha,1+\alpha/2}(Q_T)\;:\; |u|_{Q_T}^{(2+\alpha , 1+\alpha/2)}\leq M,\; u(x,0) = u_j^0(x)\},\qquad &j=1,\ldots,6 \\
    Y_j &\coloneqq \{ \theta\in C^\alpha([0,T]) \;:\;  |\theta|^{(\alpha)} \leq M, \; \theta(0) = \theta_j^0 \},\qquad &j=1,2,3. 
\end{aligned}
\end{equation}
We define an operator
\begin{equation}
\begin{split}
    \mathcal{R}: \prod_{j=1}^6 X_j \times \prod_{j=1}^3 Y_j &\to \prod_{j=1}^6 X_j \times \prod_{j=1}^3 Y_j \\
    (\overline{\u},\overline{\bftheta}) = ((\overline{u}_j)_{j=1}^6 , (\overline{\theta}_j)_{j=1}^3) &\mapsto \mathcal{R}(\overline{\u},\overline{\bftheta})
\end{split}
\end{equation}
where $(\u,\bftheta) = \mathcal{R}(\overline{\u},\overline{\bftheta})$ is the unique solution to the linearized system
\begin{equation}
\label{linearized}
\begin{split}
    u_{jt} - D_j u_{jxx} &= f_j \\
    u_j(x,0) &= u_j^0(x) \\
    \theta_j(t) &= \theta_j^0 + \int_0^t \nu\bigg[\sigma'(\overline{\theta}_{j-1}(s) -  \overline{\theta}_{j}(s))\int_0^1 |\overline{\p}_{jx}(x,s)|\;dx   \\
    &\qquad\qquad -\sigma'(\overline{\theta}_{j}(s) -  \overline{\theta}_{j+1}(s))\int_0^1 |\overline{\p}_{(j+1)x}(x,s)|\;dx \bigg]\;ds 
\end{split}
\end{equation}
with
\begin{equation}
\begin{split}
    D_j &= \begin{cases}
        \frac{ \sigma_{3,1}^0}{|\p_{1x}^0|^2} &\text{ for }j=1,2\\
        \frac{ \sigma_{1,2}^0}{|\p_{2x}^0|^2} &\text{ for }j=3,4\\
        \frac{ \sigma_{2,3}^0}{|\p_{3x}^0|^2} &\text{ for }j=5,6
    \end{cases}\\
    f_j &= \begin{cases}
        \bigg( \frac{\overline{\sigma}_{3,1}}{|\overline{\p}_{1x}|^2} -\frac{\sigma_{3,1}^0}{|\p_{1x}^0|^2} \bigg)\overline{u}_{jxx} &\text{ for }j=1,2\\
        \bigg( \frac{\overline{\sigma}_{1,2}}{|\overline{\p}_{2x}|^2} -\frac{\sigma_{1,2}^0}{|\p_{2x}^0|^2} \bigg)\overline{u}_{jxx} &\text{ for }j=3,4\\
        \bigg( \frac{\overline{\sigma}_{2,3}}{|\overline{\p}_{3x}|^2} -\frac{\sigma_{2,3}^0}{|\p_{3x}^0|^2} \bigg)\overline{u}_{jxx} &\text{ for }j=5,6.
    \end{cases}
\end{split}
\end{equation}
Here, we use the notation
\begin{equation}
\begin{split}
    \sigma_{j,k}^0 \coloneqq \sigma(\theta_j^0 - \theta_k^0)\\
    \overline{\sigma}_{j,k}\coloneqq \sigma(\overline{\theta}_j - \overline{\theta}_k).
\end{split}
\end{equation}
The system above is subject to the linearized boundary conditions
\begin{equation}
\begin{split}
    u_1(0,t) &= u_3(0,t) = u_5(0,t)\\
    &= u_1^0(0) + \frac{1}{\mu}\int_0^t \bigg( \overline{\sigma}_{3,1}(r)\frac{\overline{u}_{1x}(0,r)}{|\overline{\p}_{1x}(0,r)|} + \overline{\sigma}_{1,2}(r)\frac{\overline{u}_{3x}(0,r)}{|\overline{\p}_{2x}(0,r)|} + \overline{\sigma}_{2,3}(r)\frac{\overline{u}_{5x}(0,r)}{|\overline{\p}_{3x}(0,r)|}\bigg) \;dr\\
    &\eqqcolon\Phi_{\overline{\p},\overline{\bftheta}}(t) \\
    u_2(0,t) &= u_4(0,t) = u_6(0,t) \\
    &= u_2^0(0) + \frac{1}{\mu}\int_0^t \bigg( \overline{\sigma}_{3,1}(r)\frac{\overline{u}_{2x}(0,r)}{|\overline{\p}_{1x}(0,r)|} + \overline{\sigma}_{1,2}(r)\frac{\overline{u}_{4x}(0,r)}{|\overline{\p}_{2x}(0,r)|} + \overline{\sigma}_{2,3}(r)\frac{\overline{u}_{6x}(0,r)}{|\overline{\p}_{3x}(0,r)|}\bigg) \;dr\\
    &\eqqcolon\Psi_{\overline{\p},\overline{\bftheta}}(t) \\
    \big( u_1(1,t) , u_2(1,t) \big) &\equiv \mathbf{x}_1, \\
    \big( u_3(1,t) , u_4(1,t) \big) &\equiv \mathbf{x}_2, \\
    \big( u_5(1,t) , u_6(1,t) \big) &\equiv \mathbf{x}_3 
\end{split}
\end{equation}

\subsection{Outline of Proof of Theorem \ref{maintheorem}}

\underline{\textbf{Step 1:}} We first show that there exists a unique solution to the linearized system \eqref{linearized}. Note that in \eqref{linearized}, the third equation already defines $\theta_j(t)$ via the fundamental theorem of calculus. Thus, we only need to prove the existence and uniqueness of the solution $u_j$ using parabolic PDE theory. In contrast to \cite{bronsardreitich}, our (Dirichlet) boundary conditions for $u_j$, $j=1,\ldots,6$ are \emph{decoupled} and hence we do not need to verify the complementarity conditions for parabolic \emph{systems} \cite{solonnikov}. Instead, we only invoke the theory for \emph{scalar} linear parabolic PDEs to guarantee existence and uniqueness of solutions, together with a Schauder-type estimate of the form
\begin{equation}
\label{solonnikov}
\begin{split}
    \sum_{j=1}^6 |u_j|_{Q_T}^{(2+\alpha , 1+\alpha/2)} &\leq C_S \bigg( \sum_{j=1}^6 |f_j|_{Q_T}^{(\alpha , \alpha/2)} + \sum_{j=1}^6|u_j^0|^{(2+\alpha)} + |\Phi_{\overline{\p},\overline{\bftheta}}|^{(1+\alpha/2)}\\ 
    &+ |\Psi_{\overline{\p},,\overline{\bftheta}}|^{(1+\alpha/2)} + \sum_{j=1}^3 \sum_{k=1}^2 |x_{j,k}| \bigg)
\end{split}
\end{equation}
which will be useful later for showing that $\mathcal{R}$ is a contraction mapping.

\underline{\textbf{Step 2:}} Next, we need to show that the operator $\mathcal{R}$ is well-defined. This amounts to finding suitable conditions on the parameters (i.e. $\delta, M, T,\ldots$) which will ensure that if 
\begin{equation}
\begin{aligned}
    |\overline{u}_j|_{Q_T}^{(2+\alpha , 1+\alpha/2)}&\leq M, \qquad &j=1,\ldots,6\\
    |\overline{\theta}_j|^{(\alpha)}&\leq M, \qquad & j=1,2,3
\end{aligned}
\end{equation}
then
\begin{equation}
\begin{aligned}
    |u_j|_{Q_T}^{(2+\alpha , 1+\alpha/2)}&\leq M, \qquad &j=1,\ldots,6\\
    |\theta_j|^{(\alpha)}&\leq M, \qquad & j=1,2,3.
\end{aligned}
\end{equation}

\underline{\textbf{Step 3:}} Lastly, we establish that $\mathcal{R}$ is a contraction mapping by considering the difference of two solutions. More precisely, let $(\overline{\u},\overline{\bftheta})$, $(\overline{\vv},\overline{\bfphi}) \in\prod_{j=1}^6 X_j \times \prod_{j=1}^3 Y_j$, and $(\u,\bftheta) = \mathcal{R}(\overline{\u},\overline{\bftheta})$, $(\vv,\bfphi) = \mathcal{R}(\overline{\vv},\overline{\bfphi})$. Denote the differences by $(\overline{\w},\overline{\bfvarphi}) = (\overline{\u}-\overline{\vv},\overline{\bftheta}- \overline{\bfphi})$ and $(\w,\bfvarphi) = (\u-\vv,\bftheta- \bfphi)$. We will show that $\mathcal{R}$ is a contraction mapping by showing that
\begin{equation}
    \sum_{j=1}^6 |w_j|_{Q_T}^{(2+\alpha , 1+\alpha/2)} + \sum_{j=1}^3 |\varphi_j|^{(\alpha)}\leq \frac{1}{2}\bigg(\sum_{j=1}^6 |\overline{w}_j|_{Q_T}^{(2+\alpha , 1+\alpha/2)} + \sum_{j=1}^3 |\overline{\varphi}_j|^{(\alpha)}\bigg)
\end{equation}
(by further restricting $T$ if necessary). This will be done by considering a similar system of parabolic equations satisfied by $\overline{\w}$ and using Schauder-type estimates.

\subsection{Step 1: Existence and uniqueness of solutions via linear parabolic theory}

The operators $\frac{\partial}{\partial t} - D_j \frac{\partial^2}{\partial x^2}$ are uniformly parabolic since
\begin{equation}
    0 < \frac{\min_{j=1,2,3} \sigma_{j-1,j}^0}{\max_{j=1,2,3}\sup_{x\in[0,1]} |\p_{jx}^0 (x)|} \leq  D_j \leq \frac{\max_{j=1,2,3} \sigma_{j-1,j}^0}{\min_{j=1,2,3}\inf_{x\in[0,1]} |\p_{jx}^0 (x)|} < \infty
\end{equation}
for all $j=1,\ldots,6$. Since the coefficients $D_j(x)$ are of class $C^\alpha([0,1])$ (by our assumptions on the initial data $\p_j^0(x)$) and the compatibility conditions in Definition \ref{def:parcomp} are satisfied, there exist unique solutions $u_j \in C^{2+\alpha , 1+\alpha/2}(Q_T)$, $j=1,\ldots, 6$. Moreover, these solutions collectively satisfy the estimate
\begin{equation}
    \sum_{j=1}^6 |u_j|_{Q_T}^{(2+\alpha , 1+\alpha/2)} \leq C_S \bigg( \sum_{j=1}^6 |f_j|_{Q_T}^{(\alpha , \alpha/2)} + \sum_{j=1}^6|u_j^0|^{(2+\alpha)} + |\Phi_{\overline{\p},\overline{\bftheta}}|^{(1+\alpha/2)} + |\Psi_{\overline{\p},\overline{\bftheta}}|^{(1+\alpha/2)} + \sum_{j=1}^3 \sum_{k=1}^2 |x_{j,k}| \bigg)
\end{equation}
(see \cite{ladyzhenskaia} Ch.IV Theorem 5.2).

\subsection{Step 2: $\mathcal{R}$ is well-defined}

We shall henceforth assume $T$ to be sufficiently small (depending on $M$ and $Lip(\sigma)$) so that 
\begin{align}
    |\overline{\p}_{jx}(x,t)|>\frac{\delta}{2}
\end{align}
for all $(x,t)\in[0,1]\times[0,T]$ and
\begin{align}
    \overline{\sigma}_{j,k}(t) > 0
\end{align}
for all $t\in[0,T]$. Also, as we are only interested in local existence, we may assume for simplicity that $T\leq1$ 
so that $T^\alpha \leq T^\beta$ whenever $\beta\leq\alpha$. In the following, the constants that depend on various parameters (e.g. $C_{\delta,M,\ldots}$) may change from line to line.

\subsubsection{Estimate $|f_j|_{Q_T}^{(\alpha , \alpha/2)}$}

We will only show that
\begin{equation}
    |f_1|_{Q_T}^{(\alpha , \alpha/2)} \leq 
    C_{\delta,M,\sigma} T^{\alpha/2}
\end{equation}
since $|f_j|_{Q_T}^{(\alpha , \alpha/2)}$, $j=2,\ldots,6$ satisfy similar estimates.\\

By definition of $f_1$, we have
\begin{equation}
\begin{split}
    |f_1|_{Q_T}^{(\alpha , \alpha/2)} &= \bigg|\bigg( \frac{\overline{\sigma}_{3,1}}{|\overline{\p}_{1x}|^2} -\frac{\sigma_{3,1}^0}{|\p_{1x}^0|^2}\bigg)\overline{u}_{1xx}\bigg|_{Q_T}^{(\alpha , \alpha/2)}\\
    &\leq \bigg|\overline{\sigma}_{3,1}\bigg( \frac{1}{|\overline{\p}_{1x}|^2} -\frac{1}{|\p_{1x}^0|^2}\bigg)\overline{u}_{1xx}\bigg|_{Q_T}^{(\alpha , \alpha/2)}  + \bigg|\frac{1}{|\p_{1x}^0|^2}(\overline{\sigma}_{3,1} - \sigma_{3,1}^0)\overline{u}_{1xx}\bigg|_{Q_T}^{(\alpha , \alpha/2)}\\
    &\eqqcolon |I|_{Q_T}^{(\alpha , \alpha/2)} + |II|_{Q_T}^{(\alpha , \alpha/2)}
\end{split}
\end{equation}
Recall that 
\begin{equation}
    |\cdot|_{Q_T}^{(\alpha , \alpha/2)} = |\cdot|_{Q_T}^{(0)} + \langle \cdot \rangle_{x,Q_T}^{(\alpha)} + \langle \cdot \rangle_{t,Q_T}^{(\alpha/2)}. 
\end{equation}
We have
\begin{equation}
    |I|_{Q_T}^{(0)} \leq |\sigma|^{(0)} \bigg|\bigg( \frac{1}{|\overline{\p}_{1x}|^2} -\frac{1}{|\p_{1x}^0|^2}\bigg)\overline{u}_{1xx} \bigg|_{Q_T}^{(0)} \leq |\sigma|^{(0)} C_\delta M^2 T^{1/2} \leq C_{\delta,M,\sigma} T^{1/2}
\end{equation}
where the second inequality is estimate (43) of \cite{bronsardreitich}.

Since $\overline{\sigma}_{3,1}$ depends only on $t$, we have a similar estimate
\begin{equation}
    \langle I \rangle_{x,Q_T}^{(\alpha)} \leq C_{\delta,M,\sigma} T^{1/2}.
\end{equation}
Next,
\begin{equation}
\begin{split}
    \langle I \rangle_{t,Q_T}^{(\alpha/2)} &\leq |\sigma|^{(0)} \bigg\langle  \bigg( \frac{1}{|\overline{\p}_{1x}|^2} -\frac{1}{|\p_{1x}^0|^2}\bigg)\overline{u}_{1xx}\bigg\rangle_{t,Q_T}^{(\alpha/2)} + \bigg| \bigg( \frac{1}{|\overline{\p}_{1x}|^2} -\frac{1}{|\p_{1x}^0|^2}\bigg)\overline{u}_{1xx}\bigg|_{Q_T}^{(0)}\langle  \overline{\sigma}_{3,1}\rangle^{(\alpha/2)}\\
    &= |\sigma|^{(0)} C_\delta M^2 T^{1/2} + C_\delta M^2 T^{1/2}\langle  \overline{\sigma}_{3,1}\rangle^{(\alpha/2)}.
\end{split}
\end{equation}
By Lemma \ref{lemma-lipschitz},
\begin{align}\label{sigma-estimate}
    \langle  \overline{\sigma}_{3,1}\rangle^{(\alpha/2)}\leq Lip(\sigma) \bigg( \langle \overline{\theta}_3 \rangle^{(\alpha
    )}  +  \langle \overline{\theta}_1 \rangle^{(\alpha
    )} \bigg)T^{\alpha/2}  \leq C_{\sigma}M T^{\alpha/2},
\end{align}
which gives us
\begin{equation}
    \langle I \rangle_{t,Q_T}^{(\alpha/2)} \leq C_{\delta,\sigma}( M^2 T^{1/2} + M^3 T^{(1+\alpha)/2} ).
\end{equation}

Collecting these estimates, we get
\begin{equation}
    |I|_{Q_T}^{(\alpha , \alpha/2)} \leq C_{\delta,M,\sigma} T^{1/2}.
\end{equation}

We estimate $|II|_{Q_T}^{(\alpha , \alpha/2)}$ in a similar way. Using Lemma \ref{lemma-zero} and Lemma \ref{lemma-lipschitz}, we get
\begin{equation}
    |II|_{Q_T}^{(0)} \leq C_\delta |\overline{\sigma}_{3,1} - \sigma_{3,1}^0|^{(0)} |\overline{u}_{1xx}|_{Q_T}^{(0)} \leq C_\delta Lip(\sigma) M T^\alpha M \leq C_{\delta,M,\sigma} T^\alpha
\end{equation}
and
\begin{equation}
\begin{split}
     \langle II \rangle_{x,Q_T}^{(\alpha)} &\leq  |\overline{\sigma}_{3,1} - \sigma_{3,1}^0|^{(0)}  \bigg\langle \frac{1}{|\p_{1x}^0|^2}\overline{u}_{1xx}\bigg\rangle_{x,Q_T}^{(\alpha)}\\
     &\leq C_\sigma MT^\alpha \bigg\langle \frac{1}{|\p_{1x}^0|^2}\overline{u}_{1xx}\bigg\rangle_{x,Q_T}^{(\alpha)}.
\end{split}
\end{equation}
Note that
\begin{equation}
\begin{split}
    \bigg\langle \frac{1}{|\p_{1x}^0|^2}\overline{u}_{1xx}\bigg\rangle_{x,Q_T}^{(\alpha)} &\leq \bigg| \frac{1}{|\p_{1x}^0|^2}\bigg|^{(0)} \langle\overline{u}_{1xx}\rangle_{x,Q_T}^{(\alpha)} + \bigg\langle \frac{1}{|\p_{1x}^0|^2}\bigg\rangle_{x,Q_T}^{(\alpha)} |\overline{u}_{1xx}|_{Q_T}^{(0)}\\
    &\leq C_{\delta,M} \qquad \text{ by Lemma \ref{lemma-lipschitz}}
\end{split}
\end{equation}
and so
\begin{equation}
    \langle II \rangle_{x,Q_T}^{(\alpha)} \leq C_{\delta,M,\sigma}T^\alpha.
\end{equation}

Lastly, by Lemma \ref{lemma-zero} and Lemma \ref{lemma-lipschitz},
\begin{equation}
\begin{split}
    \langle II \rangle_{t,Q_T}^{(\alpha/2)} &\leq C_\delta \bigg( |\overline{\sigma}_{3,1} - \sigma_{3,1}^0|^{(0)} \langle\overline{u}_{1xx}\rangle_{t,Q_T}^{(\alpha/2)} + \langle \overline{\sigma}_{3,1} - \sigma_{3,1}^0 \rangle^{(\alpha/2)}|\overline{u}_{1xx}|_{Q_T}^{(0)}\bigg)\\
    &\leq C_\delta \bigg(C_{M,\sigma} T^\alpha + \langle \overline{\sigma}_{3,1} - \sigma_{3,1}^0 \rangle^{(\alpha/2)} M\bigg)\\
    &\leq C_{\delta,M,\sigma} T^{\alpha/2}.
\end{split}
\end{equation}
Thus,
\begin{equation}
    |f_1|_{Q_T}^{(\alpha , \alpha/2)} \leq 
    C_{\delta,M,\sigma} T^{\alpha/2}.
\end{equation}

\subsubsection{Estimate $|\Phi_{\overline{\p},\overline{\bftheta}}|^{(1+\alpha/2)}$}

We will show that
\begin{equation}
    |\Phi_{\overline{\p}, \overline{\bftheta}}|^{(1+\alpha/2)} \leq |u_1^0|^{(2+\alpha)} + \frac{3T|\sigma|^{(0)}}{\mu} + \frac{3|\sigma|^{(0)}}{\mu} + C_{\delta,\mu,\sigma} (MT^{\alpha/2} + M^2 T^{1/2}).
\end{equation}
In a similar way, we can also show that
\begin{equation}
    |\Psi_{\overline{\p}, \overline{\bftheta}}|^{(1+\alpha/2)} \leq |u_2^0|^{(2+\alpha)} + \frac{3T|\sigma|^{(0)}}{\mu} + \frac{3|\sigma|^{(0)}}{\mu} + C_{\delta,\mu,\sigma} (MT^{\alpha/2} + M^2 T^{1/2}).
\end{equation}

Since 
\begin{equation}
    \sup_{t\in[0,1]}\bigg|\overline{\sigma}_{3,1}(t)\frac{\overline{u}_{1x}(0,t)}{|\overline{\p}_{1x}(0,t)|} + \overline{\sigma}_{1,2}(t)\frac{\overline{u}_{3x}(0,t)}{|\overline{\p}_{2x}(0,t)|} + \overline{\sigma}_{2,3}(t)\frac{\overline{u}_{5x}(0,t)}{|\overline{\p}_{3x}(0,t)|}\bigg| \leq 3|\sigma|^{(0)},
\end{equation}
we get
\begin{equation}
    |\Phi'_{\overline{\p},\overline{\bftheta}}|^{(0)} \leq \frac{3|\sigma|^{(0)}}{\mu} 
\end{equation}
and
\begin{equation}
\begin{split}
    |\Phi_{\overline{\p},\overline{\bftheta}}|^{(0)} &\leq |u_1^0|^{(2+\alpha)}+ \frac{3T|\sigma|^{(0)}}{\mu}.\\
\end{split}
\end{equation}
Next,
\begin{equation}
    \langle \Phi'_{\overline{\p},\overline{\bftheta}}\rangle^{(\alpha/2)} \leq \frac{1}{\mu}\bigg(\bigg\langle \overline{\sigma}_{3,1}(t)\frac{\overline{u}_{1x}(0,t)}{|\overline{\p}_{1x}(0,t)|} \bigg\rangle^{(\alpha/2)} + \bigg\langle \overline{\sigma}_{1,2}(t)\frac{\overline{u}_{3x}(0,t)}{|\overline{\p}_{2x}(0,t)|} \bigg\rangle^{(\alpha/2)} + \bigg\langle \overline{\sigma}_{2,3}(t)\frac{\overline{u}_{5x}(0,t)}{|\overline{\p}_{3x}(0,t)|} \bigg\rangle^{(\alpha/2)}\bigg).
\end{equation}
We will only derive the estimates for the first term on the RHS as the last two terms can be treated similarly.
\begin{equation}
\begin{split}
    \bigg\langle \overline{\sigma}_{3,1}\frac{\overline{u}_{1x}(0,t)}{|\overline{\p}_{1x}(0,t)|} \bigg\rangle^{(\alpha/2)} & \leq \langle \overline{\sigma}_{3,1} \rangle^{(\alpha/2)} \bigg| \frac{\overline{u}_{1x}(0,t)}{|\overline{\p}_{1x}(0,t)|} \bigg|^{(0)} + |\sigma|^{(0)} \bigg\langle \frac{\overline{u}_{1x}(0,t)}{|\overline{\p}_{1x}(0,t)|} \bigg\rangle^{(\alpha/2)}\\
    &\leq C_{\sigma}M T^{\alpha/2} + |\sigma|^{(0)} \bigg\langle \frac{\overline{u}_{1x}(0,t)}{|\overline{\p}_{1x}(0,t)|} \bigg\rangle^{(\alpha/2)} \qquad \text{ by \eqref{sigma-estimate}}\\
    &\leq C_{\sigma,M} T^{\alpha/2} + |\sigma|^{(0)} C_{\delta,M}  T^{1/2}\\
    &\leq C_{\delta,M,\sigma} T^{\alpha/2}
\end{split}
\end{equation}
where the second-last inequality follows from the fact that $|\overline{u}_{1x}|_{Q_T}^{(1+\alpha , (1+\alpha)/2)} \leq M$.

\subsubsection{Estimate $|\theta_j|^{(\alpha)}$}

Recall that $\theta_j$ satisfies the integral equation
\begin{equation}
    \theta_j(t) = \theta_j^0 + \int_0^t \nu\bigg[\sigma'(\overline{\theta}_{j-1}(s) -  \overline{\theta}_{j}(s))\int_0^1 |\overline{\p}_{jx}(x,s)|\;dx - \sigma'(\overline{\theta}_{j}(s) -  \overline{\theta}_{j+1}(s))\int_0^1 |\overline{\p}_{(j+1)x}(x,s)|\;dx \bigg]\;ds.
\end{equation}
Since 
\begin{equation}
    |\overline{\p}_{jx}|\leq 2M,
\end{equation}
we have
\begin{equation}
    |\theta_j|^{(0)} \leq |\theta_j^0| + 4T\nu Lip(\sigma) M.
\end{equation}
Also,
\begin{equation}
\begin{split}
    \langle \theta_j \rangle^{(\alpha)} &\leq \sup_{t\neq s}\bigg\{ \frac{1}{|t-s|^\alpha} \int_s^t C_{M,\nu,\sigma}\;dr \bigg\}\\
    &\leq C_{M,\nu,\sigma} T^{1-\alpha}.
\end{split}
\end{equation}
Thus,
\begin{equation}
    |\theta_j|^{(\alpha)} \leq |\theta_j^0| + C_{M,\nu,\sigma} T^{1-\alpha}(T^\alpha + 1).
\end{equation}

Finally, if we let
\begin{equation}
    M\coloneqq 2C_S\bigg( \sum_{j=1}^6|u_j^0|^{ (2+\alpha)} + |u_1^0|^{(2+\alpha)} + |u_2^0|^{(2+\alpha)}+ \frac{6|\sigma|^{(0)}}{\mu} +
    \sum_{j=1}^3 \sum_{k=1}^2 |x_{j,k}|  \bigg) + 2\sum_{j=1}^3 |\theta_j^0|,
\end{equation}
we will have
\begin{equation}
    \sum_{j=1}^6 |u_j|_{Q_T}^{(2+\alpha , 1+\alpha/2)} \leq M
\end{equation}
and
\begin{equation}
    \sum_{j=1}^3 |\theta_j|^{(\alpha)} \leq M
\end{equation}
for all sufficiently small $T \leq T_1 = T_1(\delta,M,\mu,\nu,\sigma)$. This concludes step 2.

\newpage

\subsection{Step 3: $\mathcal{R}$ is a contraction mapping}

Next, we prove that for sufficiently small $T$, $\mathcal{R}$ is a contraction mapping. Let $(\overline{\u},\overline{\bftheta})$, $(\overline{\vv},\overline{\bfphi}) \in\prod_{j=1}^6 X_j \times \prod_{j=1}^3 Y_j$, and $(\u,\bftheta) = \mathcal{R}(\overline{\u},\overline{\bftheta})$, $(\vv,\bfphi) = \mathcal{R}(\overline{\vv},\overline{\bfphi})$. Denote the differences by
\begin{equation}
\begin{split}
    \overline{w}_j &\coloneqq \overline{u}_j - \overline{v}_j, \quad j=1,\ldots,6\\
    w_j &\coloneqq u_j - v_j, \quad j=1,\ldots,6\\
    \overline{\Z}_j &\coloneqq \overline{\p}_j - \overline{\q}_j, \quad j=1,2,3\\
    \Z_j &\coloneqq \p_j - \q_j, \quad j=1,2,3\\
    \overline{\varphi}_j &\coloneqq \overline{\theta}_j - \overline{\phi}_j, \quad j=1,2,3\\
    \varphi_j &\coloneqq \theta_j - \phi_j, \quad j=1,2,3
\end{split}
\end{equation}
where $\q_1 = (v_1,v_2)$, $\q_2=(v_3,v_4)$, $\q_3 = (v_5,v_6)$.
Then the $w_j$'s and $\varphi_j$'s solve the following linear system:
\begin{equation}
\begin{split}
    w_{jt} - D_j w_{jxx} &= g_j\\
    w_j(x,0) &\equiv 0\\
    \varphi_j(t) &= \nu\int_0^t \bigg[\sigma'(\overline{\theta}_{j-1}(s) 
    -  \overline{\theta}_{j}(s))\int_0^1 |\overline{\p}_{jx}(x,s)|\;dx \\
    &\qquad- \sigma'(\overline{\theta}_{j}(s) -  \overline{\theta}_{j+1}(s))\int_0^1 |\overline{\p}_{(j+1)x}(x,s)|\;dx \bigg]\\
    &\qquad-\bigg[\sigma'(\overline{\phi}_{j-1}(s) 
    -  \overline{\phi}_{j}(s))\int_0^1 |\overline{\q}_{jx}(x,s)|\;dx\\
    &\qquad- \sigma'(\overline{\phi}_{j}(s) -  \overline{\phi}_{j+1}(s))\int_0^1 |\overline{\q}_{(j+1)x}(x,s)|\;dx \bigg]\;ds
\end{split}
\end{equation}
subject to the boundary conditions
\begin{equation}
\begin{split}
    w_1(0,t) = w_3(0,t) = w_5(0,t) = \Phi_{\overline{\p},\overline{\bftheta}}(t) - \Phi_{\overline{\q},\overline{\bfphi}}(t) \quad\text{ at }x=0\\
    w_2(0,t) = w_4(0,t) = w_6(0,t) = \Psi_{\overline{\p},\overline{\bftheta}}(t) - \Psi_{\overline{\q},\overline{\bfphi}}(t) \quad\text{ at }x=0\\
    w_j(1,t) = 0 \quad\text{ at }x=1, \quad j=1,\ldots,6.
\end{split}
\end{equation}
Here,
\begin{equation}
\begin{split}
    g_j &= \begin{cases}
        \bigg( \frac{\sigma_{3,1}^{\overline{\bftheta}}}{|\overline{\p}_{1x}|^2} -\frac{\sigma_{3,1}^0}{|\p_{1x}^0|^2} \bigg)\overline{u}_{jxx} - \bigg( \frac{\sigma_{3,1}^{\overline{\bfphi}}}{|\overline{\q}_{1x}|^2} -\frac{\sigma_{3,1}^0}{|\q_{1x}^0|^2} \bigg)\overline{v}_{jxx}&\text{ for }j=1,2\\
        \bigg( \frac{\sigma_{1,2}^{\overline{\bftheta}}}{|\overline{\p}_{2x}|^2} -\frac{\sigma_{1,2}^0}{|\p_{2x}^0|^2} \bigg)\overline{u}_{jxx} - \bigg( \frac{\sigma_{1,2}^{\overline{\bfphi}}}{|\overline{\q}_{2x}|^2} -\frac{\sigma_{1,2}^0}{|\q_{2x}^0|^2} \bigg)\overline{v}_{jxx} &\text{ for }j=3,4\\
        \bigg( \frac{\sigma_{2,3}^{\overline{\bftheta}}}{|\overline{\p}_{3x}|^2} -\frac{\sigma_{2,3}^0}{|\p_{3x}^0|^2} \bigg)\overline{u}_{jxx} - \bigg( \frac{\sigma_{2,3}^{\overline{\bfphi}}}{|\overline{\q}_{3x}|^2} -\frac{\sigma_{2,3}^0}{|\q_{3x}^0|^2} \bigg)\overline{v}_{jxx} &\text{ for }j=5,6.
    \end{cases}
\end{split}
\end{equation}
Since this is a linear system of uniformly parabolic PDEs with the same coefficients as the previous one, the solution also satisfies a Schauder-type estimate
\begin{equation}
    \sum_{j=1}^6 |w_j|_{Q_T}^{(2+\alpha , 1+\alpha/2)} \leq C_S \bigg( \sum_{j=1}^6 |g_j|_{Q_T}^{(\alpha , \alpha/2)} +  |\Phi_{\overline{\p},\overline{\bftheta}} - \Phi_{\overline{\q},\overline{\bfphi}}|^{(1+\alpha/2)} + |\Psi_{\overline{\p},\overline{\bftheta}} - \Psi_{\overline{\q},\overline{\bfphi}}|^{(1+\alpha/2)}  \bigg).
\end{equation}


\subsubsection{Estimate $|g_j|_{Q_T}^{(\alpha , \alpha/2)}$}

As before, it suffices to just consider $j=1$.
\begin{equation}
\begin{split}
    |g_1|_{Q_T}^{(\alpha , \alpha/2)} &\leq \bigg|\bigg( \frac{\sigma_{3,1}^{\overline{\bftheta}}}{|\overline{\p}_{1x}|^2} -\frac{\sigma_{3,1}^0}{|\p_{1x}^0|^2} \bigg)\overline{w}_{1xx}\bigg|_{Q_T}^{(\alpha , \alpha/2)}   +   \bigg|\bigg( \frac{\sigma_{3,1}^{\overline{\bftheta}}}{|\overline{\p}_{1x}|^2} -\frac{\sigma_{3,1}^{\overline{\bfphi}}}{|\overline{\q}_{1x}|^2} \bigg)\overline{v}_{1xx}\bigg|_{Q_T}^{(\alpha , \alpha/2)}\\
    &\eqqcolon |III|_{Q_T}^{(\alpha , \alpha/2)} + |IV|_{Q_T}^{(\alpha , \alpha/2)}.
\end{split}
\end{equation}

First, we can estimate $|III|_{Q_T}^{(\alpha , \alpha/2)}$ by
\begin{equation}\label{eq:estimateIII}
\begin{split}
    |III|_{Q_T}^{(\alpha , \alpha/2)} &= \bigg|\bigg( \frac{\sigma_{3,1}^{\overline{\bftheta}}}{|\overline{\p}_{1x}|^2} - \frac{\sigma_{3,1}^0}{|\overline{\p}_{1x}|^2} + \frac{\sigma_{3,1}^0}{|\overline{\p}_{1x}|^2} 
    -\frac{\sigma_{3,1}^0}{|\p_{1x}^0|^2} \bigg)\overline{w}_{1xx}\bigg|_{Q_T}^{(\alpha , \alpha/2)}\\
    &\leq \bigg| \frac{1}{|\overline{\p}_{1x}|^2} (\sigma_{3,1}^{\overline{\bftheta}} - \sigma_{3,1}^0)\overline{w}_{1xx}\bigg|_{Q_T}^{(\alpha , \alpha/2)} + \bigg|\sigma_{3,1}^0 \bigg( \frac{1}{|\overline{\p}_{1x}|^2} -\frac{1}{|\p_{1x}^0|^2}\bigg)\overline{w}_{1xx}\bigg|_{Q_T}^{(\alpha , \alpha/2)}.\\
\end{split}
\end{equation}
Using the estimate 
\begin{equation}\label{bronsard-reitich-estimate}
    \bigg|\bigg( \frac{1}{|\overline{\p}_{1x}|^2} -\frac{1}{|\p_{1x}^0|^2}\bigg)\overline{w}_{1xx} \bigg|_{Q_T}^{(\alpha,\alpha/2)} + \bigg|\bigg( \frac{1}{|\overline{\p}_{1x}|^2} -\frac{1}{|\overline{\q}_{1x}|^2}\bigg)\overline{v}_{1xx} \bigg|_{Q_T}^{(\alpha,\alpha/2)}  \leq C_\delta M T^{1/2}|\overline{\Z}_1|_{Q_T}^{(2+\alpha , 1+\alpha/2)}
\end{equation}
from \cite{bronsardreitich} (p.374), the second term on the RHS of \eqref{eq:estimateIII} can be estimated by 
\begin{equation}
    \sigma_{3,1}^0 \bigg|\bigg( \frac{1}{|\overline{\p}_{1x}|^2} -\frac{1}{|\p_{1x}^0|^2}\bigg)\overline{w}_{1xx}\bigg|_{Q_T}^{(\alpha , \alpha/2)} \leq C_{\delta,\sigma} M T^{1/2}|\overline{\Z}_1|_{Q_T}^{(2+\alpha , 1+\alpha/2)},
\end{equation}
and the first term on the RHS of \eqref{eq:estimateIII} by
\begin{equation}
\begin{split}
    \bigg| \frac{1}{|\overline{\p}_{1x}|^2} (\sigma_{3,1}^{\overline{\bftheta}} - \sigma_{3,1}^0)\overline{w}_{1xx}\bigg|_{Q_T}^{(\alpha , \alpha/2)} &= \bigg| \frac{1}{|\overline{\p}_{1x}|^2} (\sigma_{3,1}^{\overline{\bftheta}} - \sigma_{3,1}^0)\overline{w}_{1xx}\bigg|_{Q_T}^{(0)} + \bigg\langle \frac{1}{|\overline{\p}_{1x}|^2} (\sigma_{3,1}^{\overline{\bftheta}} - \sigma_{3,1}^0)\overline{w}_{1xx}\bigg\rangle_{x,Q_T}^{(\alpha)}\\
    &\qquad+\bigg\langle \frac{1}{|\overline{\p}_{1x}|^2} (\sigma_{3,1}^{\overline{\bftheta}} - \sigma_{3,1}^0)\overline{w}_{1xx}\bigg\rangle_{t,Q_T}^{(\alpha/2)}\\
    &\leq C_{\delta,\sigma}T^\alpha |\overline{w}_1|_{Q_T}^{(2+\alpha , 1+\alpha/2)} + C_{\delta,\sigma,M}T^\alpha |\overline{w}_1|_{Q_T}^{(2+\alpha , 1+\alpha/2)} \\
    &\qquad+ C_{\delta,\sigma,M}|\overline{w}_1|_{Q_T}^{(2+\alpha , 1+\alpha/2)} (T^{\alpha/2} + T^\alpha)\\
    &\leq C_{\delta,M,\sigma}T^{\alpha/2}|\overline{\Z}_1|_{Q_T}^{(2+\alpha , 1+\alpha/2)}.
\end{split}
\end{equation}

Next, we estimate $|IV|_{Q_T}^{(\alpha , \alpha/2)}$.
\begin{equation}
\begin{split}
    |IV|_{Q_T}^{(0)} &\leq \bigg| \frac{\sigma_{3,1}^{\overline{\bftheta}}}{|\overline{\p}_{1x}|^2} -\frac{\sigma_{3,1}^{\overline{\bfphi}}}{|\overline{\q}_{1x}|^2} \bigg|_{Q_T}^{(0)} |\overline{v}_{1xx}|_{Q_T}^{(0)}\\
    &\leq M\bigg| \frac{\sigma_{3,1}^{\overline{\bftheta}}}{|\overline{\p}_{1x}|^2} -\frac{\sigma_{3,1}^{\overline{\bfphi}}}{|\overline{\q}_{1x}|^2} \bigg|_{Q_T}^{(0)}\\
    &\leq C_{\delta,M}  \bigg(|\sigma|^{(0)}  |\overline{\Z}_{1x}|_{Q_T}^{(0)}  +   Lip(\sigma) (|\overline{\varphi}_3|^{(0)} + |\overline{\varphi}_1|^{(0)})\bigg)\\
    &\leq C_{\delta,M,\sigma} \bigg( |\overline{\Z}_1|_{Q_T}^{(2+\alpha , 1+\alpha/2)}T^{\alpha/2}  +  T^\alpha (|\overline{\varphi}_3|^{(\alpha)} + |\overline{\varphi}_1|^{(\alpha)})\bigg)
\end{split}
\end{equation}
where in the last inequality we used the fact that $\overline{\Z}_j(\cdot,0)=\mathbf{0}$ and $\overline{\varphi}_j(0)=0$, together with Lemma \ref{lemma-zero}.

\begin{equation}
\begin{split}
    \langle IV\rangle_{x,Q_T}^{(\alpha)} &= \bigg\langle\bigg( \frac{\sigma_{3,1}^{\overline{\bftheta}}}{|\overline{\p}_{1x}|^2} -\frac{\sigma_{3,1}^{\overline{\bfphi}}}{|\overline{\q}_{1x}|^2} \bigg)\overline{v}_{1xx}\bigg\rangle_{x,Q_T}^{(\alpha )}\\
    &\leq \bigg\langle
     \sigma_{3,1}^{\overline{\bftheta}} \bigg(\frac{1}{|\overline{\p}_{1x}|^2} -\frac{1}{|\overline{\q}_{1x}|^2}\bigg)\overline{v}_{1xx}\bigg\rangle_{x,Q_T}^{(\alpha )} +  \bigg\langle
     \frac{1}{|\overline{\q}_{1x}|^2}(\sigma_{3,1}^{\overline{\bftheta}} - \sigma_{3,1}^{\overline{\bfphi}})\overline{v}_{1xx}\bigg\rangle_{x,Q_T}^{(\alpha )}\\
     &\leq C_{\delta,M,\sigma}T^{1/2}|\overline{\Z}_1|_{Q_T}^{(2+\alpha , 1+\alpha/2)} + Lip(\sigma) (|\varphi_3|^{(0)} + |\varphi_1|^{(0)})\bigg\langle \frac{1}{|\overline{\q}_{1xx}|^2}\overline{v}_{1xx}\bigg\rangle_{x,Q_T}^{(\alpha )}\\
     &\leq C_{\delta,M,\sigma}T^{1/2}|\overline{\Z}_1|_{Q_T}^{(2+\alpha , 1+\alpha/2)} + C_{\delta,M,\sigma}T^\alpha (|\varphi_3|^{(\alpha)} + |\varphi_1|^{(\alpha)})
\end{split}
\end{equation}
where we used \eqref{bronsard-reitich-estimate} in the second inequality. 

Estimating the seminorms in $t$, we have
\begin{equation}
\begin{split}
    \langle IV\rangle_{t,Q_T}^{(\alpha/2)} &\leq \bigg\langle
     \sigma_{3,1}^{\overline{\bftheta}} \bigg(\frac{1}{|\overline{\p}_{1x}|^2} -\frac{1}{|\overline{\q}_{1x}|^2}\bigg)\overline{v}_{1xx}\bigg\rangle_{t,Q_T}^{(\alpha/2)} +  \bigg\langle
     \frac{1}{|\overline{\q}_{1x}|^2}(\sigma_{3,1}^{\overline{\bftheta}} - \sigma_{3,1}^{\overline{\bfphi}})\overline{v}_{1xx}\bigg\rangle_{t,Q_T}^{(\alpha/2)}\\
     &\leq \langle IV_a \rangle_{t,Q_T}^{(\alpha/2)} + \langle IV_b \rangle_{t,Q_T}^{(\alpha/2)}
\end{split}
\end{equation}
with
\begin{equation}
\begin{split}
    \langle IV_a \rangle_{t,Q_T}^{(\alpha/2)} &\leq C_{\delta,\sigma}MT^{1/2} |\overline{\Z}_1|_{Q_T}^{(2+\alpha , 1+\alpha/2)}   +  \langle \sigma_{3,1}^{\overline{\bftheta}} \rangle^{(\alpha/2)}\bigg|\bigg(\frac{1}{|\overline{\p}_{1x}|^2} -\frac{1}{|\overline{\q}_{1x}|^2}\bigg)\overline{v}_{1xx}\bigg|_{Q_T}^{(0)}\\
    &\leq C_{\delta,M,\sigma}T^{1/2}|\overline{\Z}_1|_{Q_T}^{(2+\alpha , 1+\alpha/2)}  +  C_{\delta,\sigma}MT^{\alpha/2}MT^{1/2} |\overline{\Z}_1|_{Q_T}^{(2+\alpha , 1+\alpha/2)} \\
    &= C_{\delta,M,\sigma}T^{1/2}|\overline{\Z}_1|_{Q_T}^{(2+\alpha , 1+\alpha/2)}. 
\end{split}
\end{equation}
and
\begin{equation}
\begin{split}
    \langle IV_b \rangle_{t,Q_T}^{(\alpha/2)} &= \bigg\langle
     \frac{1}{|\overline{\q}_{1x}|^2}(\sigma_{3,1}^{\overline{\bftheta}} - \sigma_{3,1}^{\overline{\bfphi}})\overline{v}_{1xx}\bigg\rangle_{t,Q_T}^{(\alpha/2)}\\
     &\leq \langle \sigma_{3,1}^{\overline{\bftheta}} - \sigma_{3,1}^{\overline{\bfphi}} \rangle^{(\alpha/2)}\bigg|\frac{1}{|\overline{\q}_{1x}|^2}\overline{v}_{1xx}\bigg|_{Q_T}^{(0)} +  |\sigma_{3,1}^{\overline{\bftheta}} - \sigma_{3,1}^{\overline{\bfphi}}|^{(0)} \bigg\langle \frac{1}{|\overline{\q}_{1x}|^2}\overline{v}_{1xx} \bigg\rangle_{t,Q_T}^{(\alpha/2)}\\
     &\leq C_\sigma \bigg(\langle \varphi_3\rangle^{(\alpha/2)} + \langle \varphi_1\rangle^{(\alpha/2)}  + M(|\varphi_3|^{(0)} + |\varphi_1|^{(0)})\bigg)   C_\delta M \\
     &\qquad + Lip(\sigma) T^\alpha (|\varphi_3|^{(\alpha)} + |\varphi_1|^{(\alpha)})C_\delta (M+M^3)\\
     &\leq C_{\delta,M,\sigma}T^{\alpha/2}(|\varphi_3|^{(\alpha)} + |\varphi_1|^{(\alpha)}  ) + C_{\delta,M,\sigma}T^\alpha (|\varphi_3|^{(\alpha)} + |\varphi_1|^{(\alpha)}).
\end{split}
\end{equation}
To obtain the second inequality, we used Lemma \ref{lemma-difference} with $f(t) = \overline{\theta}_3 (t) - \overline{\theta}_1 (t) $ and $g(t) = \overline{\phi}_3 (t) - \overline{\phi}_1 (t) $. In conclusion, 
\begin{equation}
    |g_1|_{Q_T}^{(\alpha , \alpha/2)} \leq C_{\delta,M,\sigma} T^{\alpha/2}\bigg(|\overline{\Z}_1|_{Q_T}^{(2+\alpha , 1+\alpha/2)} +|\varphi_3|^{(\alpha)} + |\varphi_1|^{(\alpha)}  \bigg).
\end{equation}

\subsubsection{Estimate $|\Phi_{\overline{\p},\overline{\bftheta}} - \Phi_{\overline{\q},\overline{\bfphi}}|^{(1+\alpha/2)}$}\label{estimatephi}

We will show that
\begin{equation}
    |\Phi_{\overline{\p},\overline{\bftheta}} - \Phi_{\overline{\q},\overline{\bfphi}}|^{(1+\alpha/2)} \leq C_{\delta,M,\mu,\sigma}T^{\alpha/2}\bigg( \sum_{j=1}^3 |\overline{\Z}_j|_{Q_T}^{(2+\alpha , 1+\alpha/2)} + \sum_{j=1}^3 |\varphi_j|^{(\alpha)} \bigg).
\end{equation}
A similar estimate holds for $|\Psi_{\overline{\p},\overline{\bftheta}} - \Psi_{\overline{\q},\overline{\bfphi}}|^{(1+\alpha/2)}$.

Recall that
\begin{equation}
\begin{split}
    \Phi_{\overline{\p},\overline{\bftheta}}(t)- \Phi_{\overline{\q},\overline{\bfphi}}(t) = \frac{1}{\mu}\int_0^t & \sigma_{3,1}^{\overline{\bftheta}}(r) \frac{\overline{u}_{1x}(0,r)}{|\overline{\p}_{1x}(0,r)|}  -  \sigma_{3,1}^{\overline{\bfphi}}(r) \frac{\overline{v}_{1x}(0,r)}{|\overline{\q}_{1x}(0,r)|} \\
    &+\sigma_{1,2}^{\overline{\bftheta}}(r) \frac{\overline{u}_{3x}(0,r)}{|\overline{\p}_{2x}(0,r)|}  -  \sigma_{1,2}^{\overline{\bfphi}}(r) \frac{\overline{v}_{3x}(0,r)}{|\overline{\q}_{2x}(0,r)|}\\
    &+ \sigma_{2,3}^{\overline{\bftheta}}(r) \frac{\overline{u}_{5x}(0,r)}{|\overline{\p}_{3x}(0,r)|}  -  \sigma_{2,3}^{\overline{\bfphi}}(r) \frac{\overline{v}_{5x}(0,r)}{|\overline{\q}_{3x}(0,r)|}\;dr.
\end{split}
\end{equation}
To estimate $|\Phi'_{\overline{\p},\overline{\bftheta}} - \Phi'_{\overline{\q},\overline{\bfphi}}|^{(0)}$, note that
\begin{equation}
\begin{split}
    &\bigg| \sigma_{3,1}^{\overline{\bftheta}}(t) \frac{\overline{u}_{1x}(0,t)}{|\overline{\p}_{1x}(0,t)|}  -  \sigma_{3,1}^{\overline{\bfphi}}(t) \frac{\overline{v}_{1x}(0,t)}{|\overline{\q}_{1x}(0,t)|}\bigg|^{(0)} \\
    & \leq |\sigma|^{(0)} \bigg|\frac{\overline{u}_{1x}(0,t)}{|\overline{\p}_{1x}(0,t)|} -   \frac{\overline{v}_{1x}(0,t)}{|\overline{\q}_{1x}(0,t)|}\bigg|^{(0)} + |\sigma_{3,1}^{\overline{\bftheta}}(t) -  \sigma_{3,1}^{\overline{\bfphi}}(t)|^{(0)} \bigg|\frac{\overline{v}_{1x}(0,t)}{|\overline{\q}_{1x}(0,t)|}\bigg|^{(0)}\\
    &\leq C_{\delta,\sigma}M T^{(1+\alpha)/2}|\overline{\Z}_1|_{Q_T}^{(2+\alpha , 1+\alpha/2)} + |\sigma_{3,1}^{\overline{\bftheta}}(t) -  \sigma_{3,1}^{\overline{\bfphi}}(t)|^{(0)} \\
    &\leq C_{\delta,M,\sigma} T^{(1+\alpha)/2}|\overline{\Z}_1|_{Q_T}^{(2+\alpha , 1+\alpha/2)} + Lip(\sigma) T^\alpha (|\varphi_3|^{(\alpha)} + |\varphi_1|^{(\alpha)})
\end{split}
\end{equation}
and so
\begin{equation}
    |\Phi'_{\overline{\p},\overline{\bftheta}} - \Phi'_{\overline{\q},\overline{\bfphi}}|^{(0)} \leq C_{\delta,M,\mu,\sigma}T^{\alpha/2}\bigg( \sum_{j=1}^3 |\overline{\Z}_j|_{Q_T}^{(2+\alpha , 1+\alpha/2)} + \sum_{j=1}^3 |\varphi_j|^{(\alpha)} \bigg).
\end{equation}
We also get
\begin{equation}
    |\Phi_{\overline{\p},\overline{\bftheta}} - \Phi_{\overline{\q},\overline{\bfphi}}|^{(0)} \leq C_{\delta,M,\mu,\sigma}T^{1+\alpha/2}\bigg( \sum_{j=1}^3 |\overline{\Z}_j|_{Q_T}^{(2+\alpha , 1+\alpha/2)} + \sum_{j=1}^3 |\varphi_j|^{(\alpha)} \bigg).
\end{equation}

Lastly, to estimate $\langle \Phi'_{\overline{\p},\overline{\bftheta}} - \Phi'_{\overline{\q},\overline{\bfphi}} \rangle^{(\alpha/2)}$, it suffices to estimate
\begin{equation}
\begin{split}
    & \bigg\langle\sigma_{3,1}^{\overline{\bftheta}}(t) \frac{\overline{u}_{1x}(0,t)}{|\overline{\p}_{1x}(0,t)|}  -  \sigma_{3,1}^{\overline{\bfphi}}(t) \frac{\overline{v}_{1x}(0,t)}{|\overline{\q}_{1x}(0,t)|}\bigg\rangle^{(\alpha/2)} \\
    & \leq \bigg\langle\sigma_{3,1}^{\overline{\bftheta}}(t) \bigg( \frac{\overline{u}_{1x}(0,t)}{|\overline{\p}_{1x}(0,t)|} - \frac{\overline{v}_{1x}(0,t)}{|\overline{\q}_{1x}(0,t)|} \bigg)\bigg\rangle^{(\alpha/2)} + \bigg\langle (\sigma_{3,1}^{\overline{\bftheta}}(t) - \sigma_{3,1}^{\overline{\bfphi}}(t) )\frac{\overline{v}_{1x}(0,t)}{|\overline{\q}_{1x}(0,t)|}\bigg\rangle^{(\alpha/2)}\\
    &\leq |\sigma|^{(0)}\bigg\langle\frac{\overline{u}_{1x}(0,t)}{|\overline{\p}_{1x}(0,t)|} - \frac{\overline{v}_{1x}(0,t)}{|\overline{\q}_{1x}(0,t)|}\bigg\rangle^{(\alpha/2)} + \langle \sigma_{3,1}^{\overline{\bftheta}}(t)\rangle^{(\alpha/2)} \bigg|\frac{\overline{u}_{1x}(0,t)}{|\overline{\p}_{1x}(0,t)|} - \frac{\overline{v}_{1x}(0,t)}{|\overline{\q}_{1x}(0,t)|}\bigg|^{(0)}\\
    &+ |\sigma_{3,1}^{\overline{\bftheta}}(t) - \sigma_{3,1}^{\overline{\bfphi}}(t)|^{(0)} \bigg\langle \frac{\overline{v}_{1x}(0,t)}{|\overline{\q}_{1x}(0,t)|} \bigg\rangle^{(\alpha/2)} + \langle \sigma_{3,1}^{\overline{\bftheta}}(t) - \sigma_{3,1}^{\overline{\bfphi}}(t) \rangle^{(\alpha/2)}\bigg|  \frac{\overline{v}_{1x}(0,t)}{|\overline{\q}_{1x}(0,t)|} \bigg|^{(0)}.
\end{split}
\end{equation}
Using the estimates
\begin{equation}
\begin{split}
    \bigg\langle\frac{\overline{u}_{1x}(0,t)}{|\overline{\p}_{1x}(0,t)|} - \frac{\overline{v}_{1x}(0,t)}{|\overline{\q}_{1x}(0,t)|}\bigg\rangle^{(\alpha/2)} & \leq C_\delta |\overline{\Z}_1|_{Q_T}^{(2+\alpha , 1+\alpha/2)}MT^{1/2},\\
    \langle \sigma_{3,1}^{\overline{\bftheta}}(t)\rangle^{(\alpha/2)} & \leq C_{\delta,\sigma}MT^{\alpha/2},\\
    \bigg|\frac{\overline{u}_{1x}(0,t)}{|\overline{\p}_{1x}(0,t)|} - \frac{\overline{v}_{1x}(0,t)}{|\overline{\q}_{1x}(0,t)|}\bigg|^{(0)} & \leq T^{\alpha/2}\bigg\langle\frac{\overline{u}_{1x}(0,t)}{|\overline{\p}_{1x}(0,t)|} - \frac{\overline{v}_{1x}(0,t)}{|\overline{\q}_{1x}(0,t)|}\bigg\rangle^{(\alpha/2)},
\end{split}
\end{equation}
we obtain
\begin{equation}
\begin{split}
    &\bigg\langle\sigma_{3,1}^{\overline{\bftheta}}(t) \frac{\overline{u}_{1x}(0,t)}{|\overline{\p}_{1x}(0,t)|}  -  \sigma_{3,1}^{\overline{\bfphi}}(t) \frac{\overline{v}_{1x}(0,t)}{|\overline{\q}_{1x}(0,t)|}\bigg\rangle^{(\alpha/2)} \\
    & \leq C_{\delta,M,\sigma}T^{1/2}|\overline{\Z}_1|_{Q_T}^{(2+\alpha , 1+\alpha/2)} + C_{\delta,M,\sigma}T^{\alpha/2}(|\varphi_3|^{(\alpha)} + |\varphi_1|^{(\alpha)}).
\end{split}
\end{equation}

In conclusion,
\begin{equation}
    \langle \Phi'_{\overline{\p},\overline{\bftheta}} - \Phi'_{\overline{\q},\overline{\bfphi}} \rangle^{(\alpha/2)}  \leq C_{\delta,M,\mu,\sigma}T^{\alpha/2}\bigg( \sum_{j=1}^3 |\overline{\Z}_j|_{Q_T}^{(2+\alpha , 1+\alpha/2)} + \sum_{j=1}^3 |\varphi_j|^{(\alpha)} \bigg).
\end{equation}

\subsubsection{Estimate $|\varphi_j|^{(\alpha)}$}

\begin{equation}
\begin{split}
    \varphi_j(t) &= \nu\int_0^t \bigg[\sigma'(\overline{\theta}_{j-1}(r) -  \overline{\theta}_{j}(r))\int_0^1 |\overline{\p}_{jx}(x,r)|\;dx - \sigma'(\overline{\theta}_{j}(r) -  \overline{\theta}_{j+1}(r))\int_0^1 |\overline{\p}_{(j+1)x}(x,r)|\;dx \bigg]\\
    &-\bigg[\sigma'(\overline{\phi}_{j-1}(r) -  \overline{\phi}_{j}(r))\int_0^1 |\overline{\q}_{jx}(x,r)|\;dx - \sigma'(\overline{\phi}_{j}(r) -  \overline{\phi}_{j+1}(r))\int_0^1 |\overline{\q}_{(j+1)x}(x,r)|\;dx \bigg]\;dr\\
    & = \nu\int_0^t \bigg[\sigma'(\overline{\theta}_{j-1}(r) -  \overline{\theta}_{j}(r))\int_0^1 |\overline{\p}_{jx}(x,r)|\;dx - \sigma'(\overline{\phi}_{j-1}(r) -  \overline{\phi}_{j}(r))\int_0^1 |\overline{\q}_{jx}(x,r)|\;dx \bigg]\\
    &+\bigg[\sigma'(\overline{\phi}_{j}(r) -  \overline{\phi}_{j+1}(r))\int_0^1 |\overline{\q}_{(j+1)x}(x,r)|\;dx  -\sigma'(\overline{\theta}_{j}(r) -  \overline{\theta}_{j+1}(r))\int_0^1 |\overline{\p}_{(j+1)x}(x,r)|\;dx  \bigg]\;dr\\
    & = \nu\int_0^t V + VI \;dr
\end{split}
\end{equation}

\begin{equation}
\begin{split}
    |V| &\leq |\sigma'(\overline{\theta}_{j-1}(r) -  \overline{\theta}_{j}(r)) - \sigma'(\overline{\phi}_{j-1}(r) -  \overline{\phi}_{j}(r))|\int_0^1 |\overline{\p}_{jx}(x,r)|\;dx\\
    &+ |\sigma'(\overline{\phi}_{j-1}(r) -  \overline{\phi}_{j}(r))|\int_0^1 \big||\overline{\p}_{jx}(x,r)|-|\overline{\q}_{jx}(x,r)|\big|\;dx\\
    &\leq Lip(\sigma')(|\varphi_{j-1}|^{(\alpha)} + |\varphi_{j}|^{(\alpha)})M + Lip(\sigma) |\overline{\Z}_j|_{Q_T}^{(2+\alpha , 1+\alpha/2)}
\end{split}
\end{equation}
and similarly,
\begin{equation}
    |VI| \leq Lip(\sigma')(|\varphi_{j}|^{(\alpha)} + |\varphi_{j+1}|^{(\alpha)})M + Lip(\sigma) |\overline{\Z}_{j+1}|_{Q_T}^{(2+\alpha , 1+\alpha/2)}.
\end{equation}

Thus,
\begin{equation}
    |\varphi_j|^{(0)} \leq C_{M,\nu,\sigma} T\bigg( \sum_{j=1}^3 |\overline{\Z}_j|_{Q_T}^{(2+\alpha , 1+\alpha/2)} + \sum_{j=1}^3 |\varphi_j|^{(\alpha)} \bigg)
\end{equation}
and
\begin{equation}
    \langle \varphi_j \rangle^{(\alpha)} \leq C_{M,\nu,\sigma} T^{1-\alpha}\bigg( \sum_{j=1}^3 |\overline{\Z}_j|_{Q_T}^{(2+\alpha , 1+\alpha/2)} + \sum_{j=1}^3 |\varphi_j|^{(\alpha)} \bigg)
\end{equation}
to give us
\begin{equation}
    |\varphi_j|^{(\alpha)} \leq C_{M,\nu,\sigma}(T+T^{1-\alpha})\bigg( \sum_{j=1}^3 |\overline{\Z}_j|_{Q_T}^{(2+\alpha , 1+\alpha/2)} + \sum_{j=1}^3 |\varphi_j|^{(\alpha)} \bigg).
\end{equation}

Thus there exists $T_2 = T_2(\delta,M,\mu,\nu ,\sigma)\leq T_1$ such that for all $T\leq T_2$, we have
\begin{equation}
    \sum_{j=1}^6 |w_j|_{Q_T}^{(2+\alpha , 1+\alpha/2)} + \sum_{j=1}^3 |\varphi_j|^{(\alpha)}\leq \frac{1}{2}\bigg(\sum_{j=1}^6 |\overline{w}_j|_{Q_T}^{(2+\alpha , 1+\alpha/2)} + \sum_{j=1}^3 |\overline{\varphi}_j|^{(\alpha)}\bigg),
\end{equation}
i.e. $\mathcal{R}$ is a contraction mapping.

\section{Existence of Solutions for More General Initial Data}\label{sec:sobolev}
The existence of solutions in Theorem \ref{maintheorem} relies on the fact that the initial curves $(\p_j^0)_{j=1}^3$ are of class $C^{2+\alpha}$ and satisfy a set of compatibility conditions (\eqref{compatibility1}, \eqref{compatibility2}) involving up to second derivatives of $\p_j^0$. In this section, we would like to remove these conditions in the same way as was achieved in \cite{pluda}. Owing to the way in which we linearize the triple junction drag boundary condition, we only require that the three initial curves meet at the triple junction and do not impose any conditions on the derivatives of curves at the triple junction. In particular, there are no conditions on the angles or curvatures.

As we have briefly discussed in the introduction, the possibility of starting a flow from general initial curves (even those that may not satisfy certain angle or curvature conditions at the triple junction) may be relevant if one would like to restart and extend the flow beyond a topological change. For instance, one can picture a situation with two triple junctions colliding to form an unstable quadruple junction and then quickly separating into two new triple junctions. In the case of Herring angle conditions (with equal and constant surface tensions), the quadruple junction could have an angle configuration of 60/120/60/120 degrees, making it in that instant an irregular network. We mention that curvature flow starting from an initial network that violates one or more of the usual assumptions (e.g. having more than three curves meeting at a junction, violating angle or curvature compatibility conditions at triple junctions, etc.) is an active direction of research (e.g. \cite{pluda,ilmanen-neves-schulze,irregularnetworks}).

We first introduce some different function spaces. Let $p\in[1,+\infty)$. The space $W_p^{2,1} = W_p^{2,1}((0,1)\times (0,T) ; \R)$ consists of functions $f(x,t)$ satisfying
\begin{equation}
    ||f||_{W_p^{2,1}} \coloneqq \sum_{j=0}^2 \sum_{2r+s = j}|| \partial_t^r\partial_x^s f||_{L^p} < \infty.
\end{equation}
For $s\in(0,+\infty)$, the \emph{Sobolev–Slobodeckij space} $W_p^s = W_p^s ((a,b);\R)$ consists of functions $f(x)$ satisfying
\begin{equation}
    ||f||_{W_p^s}\coloneqq||f||_{W_p^{\lfloor s \rfloor}} + [\partial_x^{\lfloor s \rfloor} f]_{s-\lfloor s \rfloor , p} < \infty,
\end{equation}
where 
\begin{equation}
\begin{split}
    ||f||_{W_p^k} &\coloneqq \bigg(\sum_{0\leq m\leq k}||\partial_x^m f||_{L^p}^p \bigg)^{1/p}, \qquad k\in \mathbb{N}\\
    [f]_{\theta,p} &\coloneqq \bigg(\int_a^b \int_a^b \frac{|f(x) - f(y)|^p}{|x-y|^{\theta p + 1}}\;dx\;dy\bigg)^{1/p}, \qquad \theta\in(0,1).
\end{split}
\end{equation}

\begin{definition}[Solution in $W_p^{2,1}$]\label{def:sobolevsol}

Let $p\in(3,+\infty)$. We say that $(u_j(x,t))_{j=1}^6 \in W_p^{2,1}((0,1)\times(0,T);\R^6)$ is a solution to the parametric problem starting from initial data $(u_j^0(x))_{j=1}^6 \in W_p^{2 - 2/p}((0,1);\R^6)$ if 
\begin{enumerate}[label=(\roman*)]
    \item $|\p_{jx}(x,t)|\neq 0$ for all $(x,t)\in[0,1]\times [0,T]$, $j=1,2,3$,
    \item $\frac{d}{dt}(u_j(0,t)) \in L^p((0,T);\R)$ for $j=1,\ldots,6$
    \item and for almost every $x\in(0,1)$ and $t\in(0,T)$ we have
    \begin{equation}
    \label{sobolevsystem}
    \begin{aligned}
        u_{jt} &= \frac{1}{|\p_{1x}|^2}u_{jxx}, & \quad j=1,2 \\
        u_{jt} &= \frac{1}{|\p_{2x}|^2}u_{jxx}, & \quad j=3,4 \\
        u_{jt} &= \frac{1}{|\p_{3x}|^2}u_{jxx}, & \quad j=5,6 \\
        u_j(x,0) &= u_j^0(x), & \quad j=1,\ldots,6 \\
        u_j(0,t) &= u_{j+2}(0,t) = u_{j+4}(0,t), & \quad j=1,2 \\
        \frac{d}{dt}(u_1(0,t)) = \frac{1}{\mu}\bigg( \frac{1}{|\p_{1x}(0,t)|} & u_{1x}(0,t) + \frac{1}{|\p_{2x}(0,t)|}u_{3x}(0,t) + \frac{1}{|\p_{3x}(0,t)|}u_{5x}(0,t)\bigg), &\\
        \frac{d}{dt}(u_2(0,t)) = \frac{1}{\mu}\bigg( \frac{1}{|\p_{1x}(0,t)|} & u_{2x}(0,t) + \frac{1}{|\p_{2x}(0,t)|}u_{4x}(0,t) + \frac{1}{|\p_{3x}(0,t)|}u_{6x}(0,t)\bigg), &\\
        \big( u_1(1,t) , u_2(1,t) \big) &\equiv \mathbf{x}_1, & \\
        \big( u_3(1,t) , u_4(1,t) \big) &\equiv \mathbf{x}_2, & \\
        \big( u_5(1,t) , u_6(1,t) \big) &\equiv \mathbf{x}_3, & 
    \end{aligned}
    \end{equation}
\end{enumerate}

\end{definition}

By Theorem 2.4 of \cite{pluda}, the embedding $W_p^{2,1}((0,1)\times(0,T);\R) \hookrightarrow C^0([0,T]; C^{1+\alpha}([0,1]))$, $\alpha\in(0,1-3/p]$, is continuous and therefore the derivatives $\p_{jx}(x,t)$ make sense pointwise. By \cite{solonnikov}, the spatial trace operator 
\begin{equation}
\begin{split}
    W_p^{2,1}((0,1)\times(0,T);\R) &\to W_p^{\frac{1}{2} - \frac{1}{2p}}((0,T);\R)\\
    u(x,t) &\mapsto (u_x(x,t))\big|_{x=0}
\end{split}
\end{equation}
is also continuous. Thus, the RHS of the triple junction drag conditions in \eqref{sobolevsystem} are $L^p$ functions of $t$. Moreover, by the Sobolev Embedding Theorem (\cite{Triebel_1978} Theorem 4.6.1.(e)),
\begin{equation}\label{eq:embedding}
    W_p^{\frac{1}{2}-\frac{1}{2p}}((0,T))\hookrightarrow C^{\alpha}((0,T)), \qquad p\in(3,+\infty), \quad \alpha \in \bigg(0,\frac{1}{2} - \frac{3}{2p}\bigg)
\end{equation}
and thus the map $t \mapsto (u_x(x,t))\big|_{x=0}$ is actually H\"{o}lder continuous, whenever $u \in W_p^{2,1}$. However, for $u \in W_p^{2,1}$, we only know that $u_t \in L^p((0,T);L^p((0,1)))$ and thus a spatial trace of $u_t$ is not well-defined a priori. This is the reason for writing the LHS of the triple junction drag condition in \eqref{sobolevsystem} as the time derivative of the spatial trace of $u_j$, i.e. $\frac{d}{dt}(u_j(0,t))$, together with a proviso that it is an element of $L^p((0,T);\R)$ and not merely a distribution.

\begin{definition}[Compatibility for initial data in $W_p^{2 - 2/p}$]

Let $p\in(3,+\infty)$ and $\p_1^0(x) = (u_1^0(x),u_2^0(x))$, $\p_2^0(x) = (u_3^0(x),u_4^0(x))$, $\p_3^0(x) = (u_5^0(x),u_6^0(x))$, with $(u_j^0(x))_{j=1}^6 \in W_p^{2 - 2/p}((0,1);\R^6)$ and $|\p_{jx}^0(x)|\neq 0$ for all $x\in[0,1]$. We say that the initial curves $(\p_j^0(x))_{j=1}^3$ satisfy the compatibility conditions for \eqref{sobolevsystem} if 
\begin{equation}\label{eq:sobolevcomp1}
    \p_1^0(0) = \p_2^0(0) = \p_3^0(0)
\end{equation}
and
\begin{equation}\label{eq:sobolevcomp2}
    \p_1^0(1) = \x_1, \; \p_2^0(1) = \x_2, \; \p_3^0(1) = \x_3.
\end{equation}
\end{definition}

We emphasize that there are no conditions imposed on the derivatives of the initial curves at the triple junction (we only require that the three initial curves meet at the triple junction) because our linearized system \eqref{linearizedsobolev} \& \eqref{linearizedsobolevbc} only prescribes Dirichlet boundary conditions (see \cite{ladyzhenskaia} Chapter 4 Section 9). In particular, there are no restrictions on the initial angles at the triple junction. We contrast this with \cite{pluda} where the initial network is required to satisfy the Herring angle condition at the triple junction.

\begin{theorem}\label{sobolevsolution}
    Let $p \in (3,\infty)$ and $\p_j^0 \in W_p^{2 - 2/p}((0,1);\R^2)$, $j=1,2,3$ be three (parametrized) curves that form an initial network satisfying the compatibility conditions \eqref{eq:sobolevcomp1} \& \eqref{eq:sobolevcomp2}. Assume that $\sigma \equiv 1$ and 
    \begin{equation}
        \delta = \min_{j=1,2,3} \inf_{x\in[0,1]} |\p_{jx}^0 (x)|>0.
    \end{equation}
    Then there exists $M,T>0$ such that the system \eqref{sobolevsystem} has a unique solution 
    \begin{equation}
        (u_j(x,t))_{j=1}^6 \in  W_p^{2,1}((0,1)\times(0,T);\R^6)  \cap \overline{B_M}.
    \end{equation}
\end{theorem}

\begin{remark}
    Following \cite{pluda}, the solution space $W_p^{2,1}((0,1)\times(0,T);\R^6)$ is equipped with the norm
    \begin{equation}
        |||f|||_{W_p^{2,1}}\coloneqq ||f||_{W_p^{2,1}} + ||f(\cdot,0)||_{W_p^{2-2/p}((0,1))}
    \end{equation}
    and $\overline{B_M}$ refers to the ball of radius $M$ in this norm. 
\end{remark}

Before giving the proof of the theorem, we recall a lemma from \cite{pluda} that guarantees a uniform lower bound on $|\p_{jx}(x,t)|$ for a short time.

\begin{lemma}[Lemma 3.9 of \cite{pluda}]\label{lem:lowerbound}
    Let $p>3$, $M>0$ and
    \begin{equation*}
        \delta = \min_{j=1,2,3} \inf_{x\in[0,1]} |\p_{jx}^0 (x)|>0.
    \end{equation*}
    Then there exists a time $\Tilde{T}$ depending only on $\delta$ and $M$ such that for all $T\in[0,\Tilde{T}]$ and $(\p_j(x,t))_{j=1}^3 \in W_p^{2,1}((0,1)\times(0,T);\R^6)\cap\overline{B_M}$ with $\p_j(x,0) = \p_j^0(x)$, we have
    \begin{equation}
        \min_{j=1,2,3} \inf_{x\in[0,1], t\in[0,T]} |\p_{jx}(x,t)| \geq \frac{\delta}{2}.
    \end{equation}
\end{lemma}

\begin{proof}[Proof of Theorem \ref{sobolevsolution}]
    The proof of Theorem \ref{sobolevsolution} is very similar to that of Theorem 3.14 in \cite{pluda}, save for a few small modifications (since we impose the triple junction drag condition instead of the Herring angle condition in \cite{pluda}). Thus, we will just provide an outline of their proof and give only the details of the necessary modifications. Essentially, the proof follows the strategy of \cite{bronsardreitich}, which was the same strategy that was used to prove Theorem \ref{maintheorem}:
    \begin{enumerate}
        \item linearizing the system \eqref{sobolevsystem} about the initial data 
        \item proving existence and uniqueness of solutions to the linearized system using parabolic theory \cite{ladyzhenskaia} (this time the solution will be in $W_p^{2,1}$ instead of $C^{2+\alpha, 1+\alpha/2}$)
        \item proving existence of solutions to the nonlinear system \eqref{sobolevsystem} by a contraction mapping argument.
    \end{enumerate}

    Take an initial network consisting of curves $\p_j^0 \in W_p^{2- 2/p}((0,1);\R^2)$ satisfying $\p_1^0(0) = \p_2^0(0) = \p_3^0(0)$, $\p_j^0(1) = \x_j$, $j=1,2,3$. For any $f(x,t) \in L^p((0,T);L^p((0,1);\R^6))$ and $\Phi(t), \Psi(t)\in W_p^{1 - 1/(2p)}((0,T);\R)$ satisfying the compatibility conditions $\Phi(0) = \Psi(0) = u_j^0(0)$, $j=1,2$, there exists a unique solution $\{u_j\}_{j=1}^6$ to the linear system
    \begin{equation}\label{linearizedsobolev}
    \begin{split}
        u_{jt} - D_j u_{jxx} &= f_j \\
        u_j(x,0) &= u_j^0(x) 
    \end{split}
    \end{equation}
    subject to the boundary conditions
\begin{equation}\label{linearizedsobolevbc}
\begin{aligned}
    u_1(0,t) = u_3(0,t) = u_5(0,t) &= \Phi(t),\\
    u_2(0,t) = u_4(0,t) = u_6(0,t) &= \Psi(t),\\
    \big( u_1(1,t) , u_2(1,t) \big) &\equiv \mathbf{x}_1,\\
    \big( u_3(1,t) , u_4(1,t) \big) &\equiv \mathbf{x}_2,\\
    \big( u_5(1,t) , u_6(1,t) \big) &\equiv \mathbf{x}_3.
\end{aligned}
\end{equation}
We recall that the $u_j$ are the components of the curves $\p_j$ and the coefficients $D_j$ were defined in \eqref{linearized}. Moreover, the solution satisfies 
\begin{equation}
\begin{split}
    \sum_{j=1}^6 ||u_j||_{W_p^{2,1}} \leq C \bigg(  & ||f||_{L^p((0,T);L^p((0,1)))} +  ||\Phi||_{W_p^{1 - 1/(2p)}((0,T))} + ||\Psi||_{W_p^{1 - 1/(2p)}((0,T))} \\
    & + T^{1/p}\sum_{j,k=1}^3 |x_{j,k}| + \sum_{j=1}^6 ||u_j^0||_{W_p^{2- 2/p}((0,1))} \bigg).
\end{split}
\end{equation}

Similar to the proof of Theorem \ref{maintheorem}, we would like to set up a contraction mapping argument by taking 
\begin{equation}
\begin{split}
    f_j &= \begin{cases}
        \bigg( \frac{1}{|\overline{\p}_{1x}|^2} -\frac{1}{|\p_{1x}^0|^2} \bigg)\overline{u}_{jxx} &\text{ for }j=1,2\\
        \bigg( \frac{1}{|\overline{\p}_{2x}|^2} -\frac{1}{|\p_{2x}^0|^2} \bigg)\overline{u}_{jxx} &\text{ for }j=3,4\\
        \bigg( \frac{1}{|\overline{\p}_{3x}|^2} -\frac{1}{|\p_{3x}^0|^2} \bigg)\overline{u}_{jxx} &\text{ for }j=5,6
    \end{cases}
\end{split}
\end{equation}
and
\begin{equation}
\begin{split}
    \Phi_{\overline{\p}}(t) &= u_1^0(0) + \frac{1}{\mu}\int_0^t \bigg( \frac{\overline{u}_{1x}(0,r)}{|\overline{\p}_{1x}(0,r)|} + \frac{\overline{u}_{3x}(0,r)}{|\overline{\p}_{2x}(0,r)|} + \frac{\overline{u}_{5x}(0,r)}{|\overline{\p}_{3x}(0,r)|}\bigg) \;dr\\
    \Psi_{\overline{\p}}(t) &= u_2^0(0) + \frac{1}{\mu}\int_0^t \bigg( \frac{\overline{u}_{2x}(0,r)}{|\overline{\p}_{1x}(0,r)|} + \frac{\overline{u}_{4x}(0,r)}{|\overline{\p}_{2x}(0,r)|} + \frac{\overline{u}_{6x}(0,r)}{|\overline{\p}_{3x}(0,r)|}\bigg) \;dr.
\end{split}
\end{equation}
for some given $\overline{\p}_j(x,t) \in W_p^{2,1}$. (This means solving system \eqref{linearizedsobolev} for a given RHS involving $\overline{\p}_j$, just like we did in \eqref{linearized}.) Note that the compatibility conditions for \eqref{linearizedsobolev} \& \eqref{linearizedsobolevbc} are satisfied since $\Phi_{\overline{\p}}(0) = u_1^0(0)$, $\Psi_{\overline{\p}}(0) = u_2^0(0)$ and the initial data $(u_j^0(x))_{j=1}^6$ were assumed to satisfy the compatibility conditions \eqref{eq:sobolevcomp1}.

Letting $\mathcal{S}$ denote the operator that maps $\overline{\p}_j$ to the solution $\p_j \in W_p^{2,1}$ of \eqref{linearizedsobolev}, we observe that a fixed point of $\mathcal{S}$ will be a solution to the nonlinear system \eqref{sobolevsystem}. Thus, in order to invoke the contraction mapping theorem, we need to check that the operator $\mathcal{S}$ is well-defined, that it maps a suitable closed ball $\overline{B_M}$ into itself and that it is a contraction mapping. Most of these have already been established in \cite{pluda} and all we have to show are the following lemmas, which mainly concern the triple junction drag condition.

\begin{lemma}
    Let $(\overline{\p}_j)_{j=1}^3 \in W_p^{2,1}((0,1)\times(0,T);\R^6)$ and define $\Phi_{\overline{\p}}(t)$, $\Psi_{\overline{\p}}(t)$ by
    \begin{equation}
    \begin{split}
        \Phi_{\overline{\p}}(t) &= u_1^0(0) + \frac{1}{\mu}\int_0^t \bigg( \frac{\overline{u}_{1x}(0,r)}{|\overline{\p}_{1x}(0,r)|} + \frac{\overline{u}_{3x}(0,r)}{|\overline{\p}_{2x}(0,r)|} + \frac{\overline{u}_{5x}(0,r)}{|\overline{\p}_{3x}(0,r)|}\bigg) \;dr\\
        \Psi_{\overline{\p}}(t) &= u_2^0(0) + \frac{1}{\mu}\int_0^t \bigg( \frac{\overline{u}_{2x}(0,r)}{|\overline{\p}_{1x}(0,r)|} + \frac{\overline{u}_{4x}(0,r)}{|\overline{\p}_{2x}(0,r)|} + \frac{\overline{u}_{6x}(0,r)}{|\overline{\p}_{3x}(0,r)|}\bigg) \;dr.
    \end{split}
    \end{equation}
Then $\Phi_{\overline{\p}}$, $\Psi_{\overline{\p}} \in W_p^{1 - 1/(2p)}((0,T);\R)$.

\end{lemma}

\begin{proof}
    We will only prove the statement for $\Phi_{\overline{\p}}$ since the proof for $\Psi_{\overline{\p}}$ is the same. 

    For every $t\in [0,T]$,
    \begin{equation}
    \begin{split}
        |\Phi_{\overline{\p}}(t)| &\leq |u_1^0 (0)| + \frac{1}{\mu}\int_0^t 3\;dr\\
        & \leq |u_1^0 (0)| + \frac{3T}{\mu}
    \end{split}
    \end{equation}
    and so $\Phi_{\overline{\p}} \in L^p((0,T);\R)$.

    Next,
    \begin{equation}
    \begin{split}
        [\Phi_{\overline{\p}}]_{1 - 1/(2p),p} &= \bigg( \int_0^T \int_0^T \frac{|\Phi(t) - \Phi(s)|^p}{|t-s|^{(1 - 1/(2p))p + 1}}\;ds\;dt \bigg)^{1/p}\\
        &\leq \bigg( \int_0^T \int_0^T \frac{\big(\frac{3}{\mu}|t-s|\big)^p}{|t-s|^{p + 1/2}}\;ds\;dt \bigg)^{1/p}\\
        & = \frac{3}{\mu} \bigg( \int_0^T \int_0^T |t-s|^{-1/2}\;ds\;dt \bigg)^{1/p}\\
        &= \frac{3}{\mu}\bigg( \frac{8}{3}T^{3/2} \bigg)^{1/p}\\
        &<\infty.
    \end{split}
    \end{equation}
\end{proof}

\begin{lemma}
    Let $\overline{\p} = (\overline{\p}_j)_{j=1}^3, \overline{\q} = (\overline{\q}_j)_{j=1}^3 \in W_p^{2,1}((0,1)\times(0,T);\R^6)$ with $\overline{\q}_j(x,0) = \overline{\p}_j(x,0) = \p_j^0(x)$, $j=1,2,3$. By Lemma \ref{lem:lowerbound}, we can choose $T>0$ sufficiently small so that $|\overline{\p}_{jx}(x,t)|, |\overline{\q}_{jx}(x,t)|>\frac{\delta}{2}$ for all $(x,t)\in[0,1]\times[0,T]$.  
    Then there exists $\beta>0$ such that
    \begin{equation}
        ||\Phi_{\overline{\p}} - \Phi_{\overline{\q}}||_{W_p^{1 - 1/(2p)}} \leq C_{\delta,\mu} T^\beta |||\overline{\p} - \overline{\q}|||_{W_p^{2,1}}.
    \end{equation}
\end{lemma}

\begin{proof}
    Similar to the estimates in section \ref{estimatephi}, we get
    \begin{equation}
        |\Phi_{\overline{\p}}(t) - \Phi_{\overline{\q}}(t)| \leq C_{\delta,\mu}T |||\overline{\p} - \overline{\q}|||_{W_p^{2,1}}
    \end{equation}
    and so 
    \begin{equation}
        ||\Phi_{\overline{\p}} - \Phi_{\overline{\q}}||_{L^p((0,T))} \leq C_{\delta,\mu}T^{1+1/p} |||\overline{\p} - \overline{\q}|||_{W_p^{2,1}}.
    \end{equation}

    Next,
    \begin{equation}
    \begin{split}
        [\Phi_{\overline{\p}} - \Phi_{\overline{\q}}]_{1 - 1/(2p),p} &= \bigg(\int_0^T \int_0^T |t-s|^{-p-1/2}|\Phi_{\overline{\p}}(t) - \Phi_{\overline{\q}}(t) - \Phi_{\overline{\p}}(s) + \Phi_{\overline{\q}}(s)|^p \;ds\;dt \bigg)^{1/p}.
    \end{split}
    \end{equation}
    Using the estimate
    \begin{equation}
        |\Phi_{\overline{\p}}(t) - \Phi_{\overline{\q}}(t) - \Phi_{\overline{\p}}(s) + \Phi_{\overline{\q}}(s)| \leq |t-s|C_{\delta,\mu}|||\overline{\p} - \overline{\q}|||_{W_p^{2,1}},
    \end{equation}
    we get
    \begin{align*}
        [\Phi_{\overline{\p}} - \Phi_{\overline{\q}}]_{1 - 1/(2p),p} &\leq C_{\delta,\mu}|||\overline{\p} - \overline{\q}|||_{W_p^{2,1}}\bigg( \int_0^T \int_0^T |t-s|^{-1/2}\;ds\;dt \bigg)^{1/p}\\
        &\leq C_{\delta,\mu}T^{3/(2p)}|||\overline{\p} - \overline{\q}|||_{W_p^{2,1}}
    \end{align*}

    Thus, 
    \begin{align*}
        ||\Phi_{\overline{\p}} - \Phi_{\overline{\q}}||_{W_p^{1 - 1/(2p)}} \leq C_{\delta,\mu} T^\beta |||\overline{\p} - \overline{\q}|||_{W_p^{2,1}}
    \end{align*}
    with $\beta = \frac{3}{2p}$.

\end{proof}

The two preceding lemmas, together with Proposition 3.10 - 3.12 of \cite{pluda} show that the operator $\mathcal{S}$ is indeed a contraction. Then, by an argument similar to Proposition 3.13 of \cite{pluda}, we see that $\mathcal{S}$ maps a closed ball into itself. Lastly, we conclude the existence and uniqueness of a solution to \eqref{sobolevsystem} using the contraction mapping theorem.

\end{proof}

\begin{remark}
    To verify the triple junction drag condition of \eqref{sobolevsystem}, observe that if $u_j(x,t) \in W_p^{2,1}$ is the fixed point of the contraction mapping, then 
    \begin{equation}
    \begin{split}
        u_1(0,t) &= u_1^0(0) + \frac{1}{\mu}\int_0^t \bigg( \frac{u_{1x}(0,r)}{|\p_{1x}(0,r)|} + \frac{u_{3x}(0,r)}{|\p_{2x}(0,r)|} + \frac{u_{5x}(0,r)}{|\p_{3x}(0,r)|}\bigg) \;dr\\
        u_2(0,t) &= u_2^0(0) + \frac{1}{\mu}\int_0^t \bigg( \frac{u_{2x}(0,r)}{|\p_{1x}(0,r)|} + \frac{u_{4x}(0,r)}{|\p_{2x}(0,r)|} + \frac{u_{6x}(0,r)}{|\p_{3x}(0,r)|}\bigg) \;dr.
    \end{split}
    \end{equation}
    By \eqref{eq:embedding}, the integrands above are continuous and so $u_j(0,t)$ are in fact differentiable and satisfy the triple junction drag condition.

\end{remark}

\section{Self-intersections During the Flow}\label{intersection}

In the case of motion of planar networks with equal {\em reduced mobilities} (see e.g. \cite{gs}) and the Herring angle condition at the triple junctions, it is believed that the only topological changes that can occur are junction-junction collisions (disappearance of one or more of the curves as they shrink to naught).
In particular, it is expected -- and supported by various rigorous theorems \cite{mantegazzanovagatortorelli,multiplejunctions} -- that two curves in the network can only come into contact at their end points.
This property of the exact dynamics has profound implications for the design of numerical methods, as it makes it feasible to attempt a classification of singular events and handle them computationally, thus allowing an extension of front tracking to be used.
(The situation is very different in higher dimensions, and even in the plane when not all reduced mobilities in the network are the same, typically requiring the use of numerical methods such as phase-field, level sets, or threshold dynamics that represent interfaces implicitly).

In this Section, we carefully point out that when the Herring angle condition at junctions is replaced with triple junction drag condition (\ref{drag}), there can be additional singularities in the form of junctions colliding with the ``interior'' of curves in a planar network, even when all reduced mobilities are equal. 
This implies, in particular, that a front tracking approach to approximating the planar curve shortening flow in the presence of triple junction drag, even under the simplest choice of surface tensions and mobilities, may be especially challenging, suggesting the use of implicit representations e.g. phase-field instead.
It also means a front tracking implementation for the Herring case cannot easily be adapted to the case with triple junction drag (else one might miss certain topological changes).
We now demonstrate all this with an example.
\begin{figure}[h]
\includegraphics[scale=0.25]{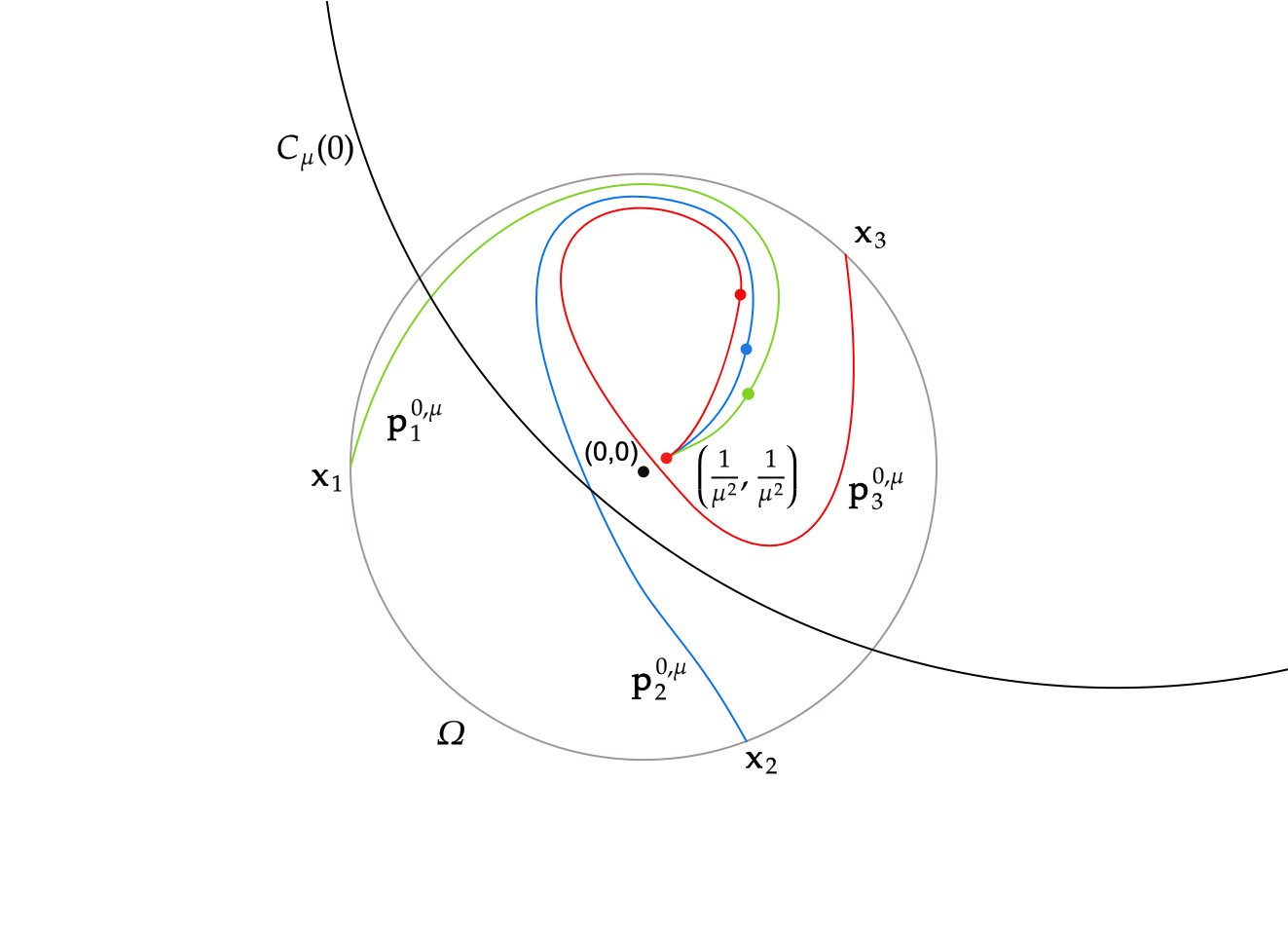}
\centering
\caption{Example of an initial network that will self-intersect under the flow}
\label{fig:intersection}
\end{figure}

Fix a large ball centered at the origin $\Omega = B(\mathbf{0},R)$ as our domain. For each $\mu\gg 1$, define three initial curves $\p_j^{0,\mu}(x): [0,1]\to\overline{\Omega}$ by
\begin{equation}
\begin{split}
    \p_j^{0,\mu}(x) \coloneqq \begin{cases}
    \bigg(\frac{1}{\mu^2} + x + \frac{1+j^2}{2\mu}\bigg(\frac{1}{\sqrt{2}} + \frac{1}{\sqrt{5}} + \frac{1}{\sqrt{10}}\bigg)x^2
    , \frac{1}{\mu^2} + jx + \frac{1+j^2}{2\mu}\bigg(\frac{1}{\sqrt{2}} + \frac{2}{\sqrt{5}} + \frac{3}{\sqrt{10}}\bigg)x^2\bigg), & x\in[0,\frac{1}{2})\\
    \gamma_j^\mu(x), & x\in(\frac{1}{2},1]
    \end{cases}
\end{split}
\end{equation}
where $\gamma_j^\mu(x)$ is any smooth continuation of the curve from $\p_j^{0,\mu}(\frac{1}{2})$ to $\p_j^{0,\mu}(1)= \x_j$ (see Figure \ref{fig:intersection}). Choose $\gamma_j^\mu(x)$ so that $\gamma_{jxx}^\mu (1) = \mathbf{0}$. We list some key properties of $\p_j^{0,\mu}$:
\begin{enumerate}[label=(\roman*)]
    \item The triple junction is at $(\frac{1}{\mu^2},\frac{1}{\mu^2})$ and the fixed endpoints are at $\x_j\in\partial B(\mathbf{0},R)$.
    \item $\p_j^{0,\mu}$ satisfies the parametric compatibility conditions at $x=0,1$.
    \item For each $\mu$, the closed curve formed by the arc $\p_3^{0,\mu}([1/2,1))$ and the line segment from $\x_3$ to $\p_3^{0,\mu}(1/2)$ (denoted by $[\x_3 , \p_3^{0,\mu}(1/2)]$) encloses the triple junction $\p_3^{0,\mu}(0)$, with a winding number of +1.
    \item We can choose $\gamma_j^\mu (x)$ in a way that ensures a uniform  H\"{o}lder bound
    \begin{equation}\label{uniformupperbound}           
        \sup_{\mu}|\p_j^{0,\mu} |^{(2+\alpha)} \leq C<\infty,
    \end{equation}
    \item and a uniform lower bound
    \begin{equation}\label{uniformlowerbound}
        \inf_{\mu}\min_{j=1,2,3}\inf_{x\in[0,1]}|\p_{jx}^{0,\mu}|\geq \delta >0.
    \end{equation}
    \item The point $\p_3^{0,\mu}(\frac{1}{2})$ is bounded away from the triple junction $(\frac{1}{\mu^2},\frac{1}{\mu^2})$, uniformly in $\mu$.
\end{enumerate}

By Theorem \ref{maintheorem}, each initial data $\p_j^{0,\mu}(x)$ generates a unique solution $\p_j^\mu(x,t)$ to the parametric problem. Thanks to \eqref{uniformupperbound}, \eqref{uniformlowerbound} and the dependence of $T$ on the various parameters in Theorem \ref{maintheorem}, there exists a positive time $\Tilde{T}>0$ such that the solutions $\p_j^\mu(x,t)$ exist on $[0,\Tilde{T}]$ for all sufficiently large $\mu$. Moreover, the solutions satisfy a uniform bound
\begin{equation}\label{eq:uniformbound}
    \limsup_{\mu\to\infty} |\p_j^\mu (x,t)|_{Q_{\Tilde{T}}}^{(2+\alpha, 1+\alpha/2)}\leq \Tilde{C}<\infty.
\end{equation}
By choosing $\Tilde{T}$ sufficiently small, we can ensure that the triple junction avoids the line segment $[\x_3 , \p_3^{\mu}(1/2,t)]$ for all $t \in [0,\Tilde{T}]$.

\begin{proposition}
    There exists $\mu>0$ such that the curve $\p_3^\mu(x,t)$ will intersect itself at some time $t_\mu \in (0,\Tilde{T})$, in the form of a `triple junction - interior' collision, i.e. 
    \begin{equation}
        \p_3^\mu(0,t_\mu) = \p_3^\mu(x,t_\mu), \qquad \text{ for some } x \in (0,1).
    \end{equation}
\end{proposition}

\begin{figure}[h]
\includegraphics[scale=0.25]{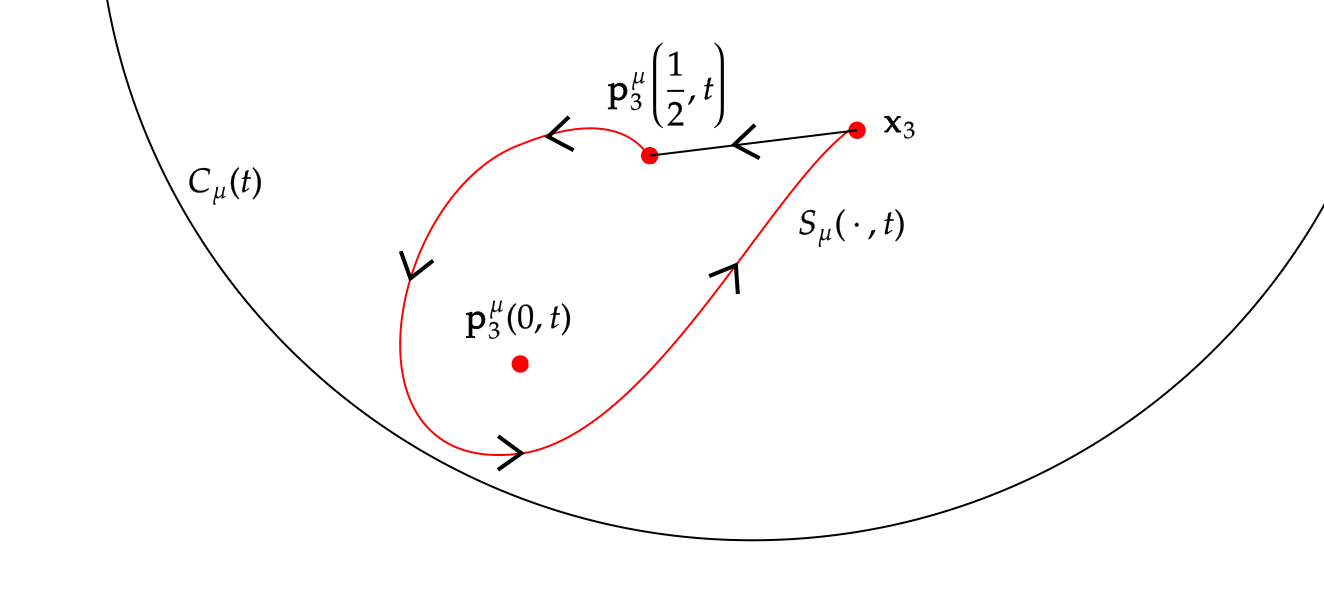}
\centering
\caption{The closed curve $S_\mu$ during the flow, enclosing the triple junction.}
\label{fig:winding}
\end{figure}

\begin{proof}
Our proof exploits the fact that the triple junction drag condition implies that the speed of the triple junction will always be bounded above by $\frac{3}{\mu}$ and is therefore very slow for large $\mu$. We first define a barrier curve to which we will apply the strong maximum principle. Let $C_{\mu}(0)$ be an initial curve that is a large circle with radius $2R$ centered at the point $(\sqrt{2}R - \frac{1}{\mu^2},\sqrt{2}R - \frac{1}{\mu^2})$. Let $C_{\mu}(t)$ be the evolution of the initial circle by curvature motion. Since $C_\mu(0)$ passes through the point $(- \frac{1}{\mu^2},- \frac{1}{\mu^2})$, the initial distance between the circle and the triple junction goes to zero as $\mu \to \infty$, i.e.
\begin{equation}
    \text{dist}(C_\mu(0), \p_3^{0,\mu}(0)) = \mathcal{O}\bigg(\frac{1}{\mu^2}\bigg), \qquad \mu \to \infty.
\end{equation}
Since $|\p_{3t}^{\mu}(0,t)| \leq \frac{3}{\mu}$, for all sufficiently large $\mu$, there exists a time $\Tilde{t}_{\mu} \in (0,\Tilde{T})$ when the triple junction will intersect the circle, with $\Tilde{t}_{\mu} \to 0$ as $\mu \to \infty$. Using the uniform bound \eqref{eq:uniformbound}, for all sufficiently small $\Tilde{t}_{\mu}$ the point $\p_3^{\mu}(1/2,t)$ will be bounded away from the circle for all $t \in [0,\Tilde{t}_{\mu}]$. 

Define an evolving closed curve $S_{\mu}(x,t)$, $(x,t)\in[0,1]\times[0,\Tilde{t}_{\mu}]$ by 
\begin{equation}
    S_{\mu}(x,t) = \begin{cases}
        \p_3^\mu (x+\frac{1}{2},t), & x \in [0,\frac{1}{2})\\
        (1-2(x-\frac{1}{2}))\x_3 + 2(x-\frac{1}{2})\p_3^\mu (\frac{1}{2},t), & x\in (\frac{1}{2},1].
    \end{cases}
\end{equation}
Thus, $S_{\mu}$ is the concatenation of the arc $\p_3^\mu ([1/2,1),t)$ with the line segment $[\x_3 , \p_3^\mu (1/2,t)]$ (Figure \ref{fig:winding}). Initially, $S_\mu(\cdot,0)$ encloses the triple junction $\p_3^\mu(0,0)$ with winding number $\text{wind}(S_\mu(\cdot,0),\p_3^\mu(0,0)) = +1$.

At time $t = \Tilde{t}_{\mu}$, the triple junction $\p_3^\mu(0,\Tilde{t}_{\mu})$ lies on the circle $C_\mu(\Tilde{t}_{\mu})$. On the other hand, during the evolution from $t=0$ to $t = \Tilde{t}_{\mu}$, the closed curve $S_\mu(t)$ must remain nested inside $C_\mu(t)$ because the comparison principle (see Theorem \ref{thm:strongmaxprinciple}) prevents the open arc $\p_3^\mu((1/2,1),t)$ from intersecting $C_\mu(t)$, and the line segment $[\x_3, \p_3^\mu(1/2,t)]$ is always bounded away from $C_\mu(t)$. This means that $\text{wind}(S_\mu(\cdot,\Tilde{t}_{\mu}) , \p_3^\mu(0,\Tilde{t}_{\mu})) = 0$, which implies that the triple junction must have collided with the closed curve $S_\mu$ at some time $t_\mu \in (0,\Tilde{t}_{\mu}]$. Indeed, if the triple junction $\p_{3}^\mu(0,t)$ never collides with $S_\mu$, then 
\begin{equation}
    \gamma(x,t) \coloneqq S_\mu(x,t) - \p_{3}^\mu(0,t), \qquad (x,t) \in [0,1]\times [0,\Tilde{t}_\mu]
\end{equation}
gives a homotopy through closed curves in $\R^2 \setminus \{0\}$, which would imply that $\text{wind}(S_\mu(\cdot,\Tilde{t}_{\mu}) , \p_3^\mu(0,\Tilde{t}_{\mu})) = \text{wind}(S_\mu(\cdot,0) , \p_3^\mu(0,0))$.

Since the triple junction avoids the line segment $[\x_3 , \p_3^{\mu}(1/2,t)]$ for all $t \in [0,\Tilde{T}]$, it must have collided with a point on the open arc $\p_3^\mu((1/2,1),t_\mu)$.

\end{proof}

\section{Stationary Solutions}
In this section, we study stationary states and their stability for our model (\ref{pde}) under certain choices of surface tension $\sigma$ that are common in existing materials science literature (e.g. the surface tension model of Read \& Shockley \cite{read_shockley}), and compare to corresponding statements in \cite{ep1,ep2}.
The upshot is that our results differ from those of \cite{ep1,ep2}; we attribute this difference to the different crystallographic orientation and surface tension choices made in those earlier works.

\begin{figure}[h]
\includegraphics[scale=0.25]{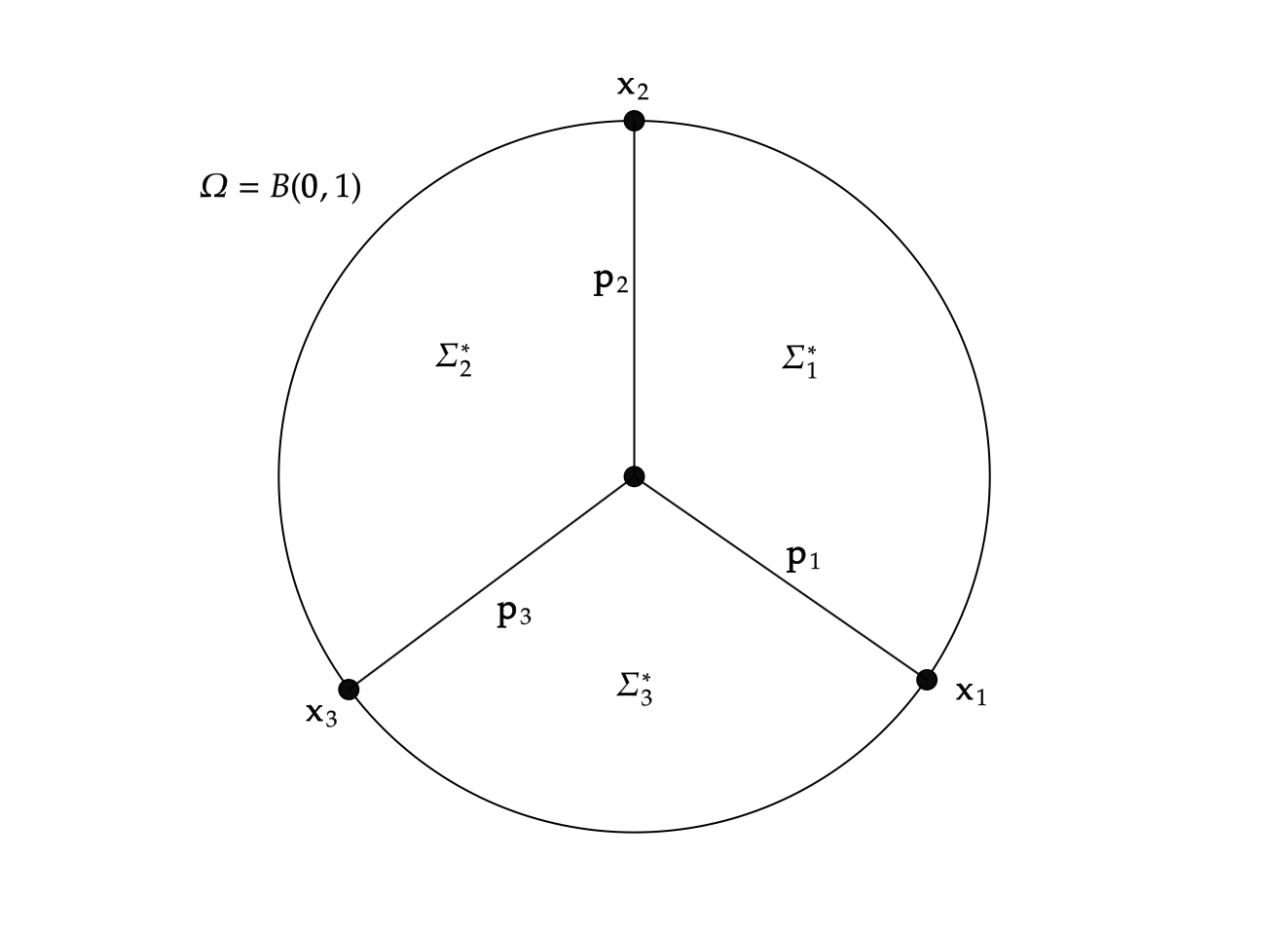}
\centering
\caption{Stationary configuration}
\label{fig:stationary}
\end{figure}

Let $\sigma$ be a function satisfying \eqref{S1}, \eqref{S2} and \eqref{S3}. To fix ideas, let
\begin{equation}
\begin{split}
    \x_1 &= \bigg(\frac{\sqrt{3}}{2}, -\frac{1}{2}\bigg) \\
    \x_2 &= (0,1)\\
    \x_3 &= \bigg(-\frac{\sqrt{3}}{2}, -\frac{1}{2}\bigg)
\end{split}
\end{equation}
so that they form an equilateral triangle centered at the origin, with $|\x_j|=1$ for $j=1,2,3$ as shown in Figure \ref{fig:stationary}.
Let 
\begin{equation}
\label{eq:stationary1}
\begin{split}
    \theta_1(t) &\equiv 0\\
    \theta_2(t) &\equiv \frac{2\pi}{3}\\
    \theta_3(t) &\equiv \frac{4\pi}{3}
\end{split}
\end{equation}
be constant orientation angles and
\begin{equation}
\label{eq:stationary2}
    \p_j(x,t) = x \x_j, \qquad x\in [0,1],\; t\geq 0, \; j=1,\ldots,6
\end{equation}
be (stationary) line segments connecting the origin (which is a triple junction) to $\x_j$.\\

To show that $(\theta_j)_{j=1}^3$, $(\p_j)_{j=1}^3$ form a stationary solution to our system, it suffices to show that the function 
\begin{equation}
\begin{split}
    (\phi_1, \phi_2, \phi_3) &\mapsto \sum_{j=1}^3 \sigma(\phi_{j-1}-\phi_j) \underbrace{\int_0^1 |\p_{jx}(x,t)|\;dx}_{=1}\\
    &= \sum_{j=1}^3 \sigma(\phi_{j-1}-\phi_j)
\end{split}
\end{equation}
has a stationary point at $(\phi_1, \phi_2, \phi_3) = (\theta_1,\theta_2,\theta_3)$.
By translating $\phi_2$ and $\phi_3$ if necessary, we may assume that $\phi_1=0$. Then, if we let $x=\phi_2$ and $y=\phi_3$ and set
\begin{equation}
    g(x,y) = \sigma(y) + \sigma(-x) + \sigma(x-y),
\end{equation}
we see that
\begin{equation}
    \nabla g (x,y) = 0 \iff \sigma'(y) = \sigma'(x-y) = -\sigma'(x).
\end{equation}

Using \eqref{S1} and \eqref{S2}, we check that
\begin{equation}
    \nabla g\bigg(\frac{2\pi}{3} , \frac{4\pi}{3}\bigg) = 0
\end{equation}
and so $(\theta_j)_{j=1}^3$, $(\p_j)_{j=1}^3$ form a stationary solution to our system.\\

Already at this stage we note some differences from \cite{ep1,ep2}.
Indeed, in \cite{ep1,ep2} the authors find that the only stationary solution to their model for grain boundary motion with triple junction drag that assumes all interfaces are straight lines is the trivial solution: There is a constant $C$ such that $\theta_j(t) = C$ for all $j\in\{1,2,3\}$ and $t\geq 0$ (hence the grains are aligned, with zero misorientation between them).
Noting that the solution (\ref{eq:stationary1}-\ref{eq:stationary2}) is also a solution to the system in \cite{ep1,ep2} for the same choice of surface tension, we attribute this discrepancy to the way misorientation angles are calculated in \cite{ep1,ep2} without $2\pi$-periodicity of the orientation angles.

We now consider the stability of our stationary solution (\ref{eq:stationary1}-\ref{eq:stationary2}). Let $\bftheta = (\theta_1,\theta_2,\theta_3)$ be the stationary angles defined in (\ref{eq:stationary1}) and let $\mathcal{N}$ denote the stationary network composed of the line segments $(\p_j)_{j=1}^3$ defined in (\ref{eq:stationary2}). In the following Proposition, we demonstrate that for a class of strictly convex surface tensions $\sigma$, the stationary solution $(\mathcal{N},\bftheta)$ is actually a stable local minimizer (modulo a translation of all three grain orientation angles at the same time) of the total surface energy $E$ defined in (\ref{eq:surface-energy}).

\begin{proposition}\label{prop:stability}
    Let $\sigma$ be a function satisfying \eqref{S1}, \eqref{S2} and \eqref{S3}. Suppose that
    \begin{equation}
        \sigma(\theta) = \theta^2 + c
    \end{equation}
    for some $c>0$ and $\theta \in [0,\pi]$. Then there exists $\overline{c}>0$ such that for all $c\geq \overline{c}$,
    \begin{equation}
        E(\mathcal{N},\bftheta) \leq E(\mathcal{M},\bfphi)
    \end{equation}
    for all networks $\mathcal{M}$ (with the same fixed endpoints $\x_j$, $j=1,2,3$ as $\mathcal{N})$ whose triple junction is sufficiently close to that of $\mathcal{N}$ and all $\bfphi = (\phi_1,\phi_2,\phi_3)$ sufficiently close to $\bftheta$. 
    
    Moreover, the stationary solution $(\mathcal{N},\bftheta)$ is stable (modulo a translation of all three grain orientation angles at the same time) in the sense that for all sufficiently small $\varepsilon>0$, there exists $\delta>0$ such that if $((\q_j(x,t))_{j=1}^3 , (\phi_j(t))_{j=1}^3)$, $(x,t)\in[0,T) \times [0,1]$ is a solution to the parametric system in Definition \ref{def:parcomp}, with initial data $((\q_j^0(x))_{j=1}^3 , (\phi_j^0)_{j=1}^3)$ satisfying
    \begin{equation}
    \begin{split}
        |(\phi_{j-1}^0 - \phi_j^0) - (\theta_{j-1} - \theta_j)|<\delta, &\qquad j=1,2,3,\\
        |\mathrm{Length}(\q_j^0) - \mathrm{Length}(\p_j)|<\delta, &\qquad j=1,2,3,\\
        |\q_j^0(0)|<\delta, &\qquad j=1,2,3,
    \end{split}
    \end{equation}
    then
    \begin{equation}
    \begin{split}
        \sup_{t\in[0,T)} d_H(\q_j(t),\p_j) < \varepsilon, &\qquad j=1,2,3,\\
        \sup_{t\in[0,T)} |(\phi_{j-1}(t) - \phi_j(t)) - (\theta_{j-1}-\theta_j)|<\varepsilon, &\qquad j=1,2,3,
    \end{split}
    \end{equation}
    where $d_H(\q_j(t),\p_j)$ denotes the Hausdorff distance between the curves $\q_j(t)$ and $\p_j$.
\end{proposition}

\begin{proof}
    Let $\mathcal{M}$ be a network composed of curves $(\q_j(x))_{j=1}^3$, $x\in[0,1]$ and $\bfphi = (\phi_1,\phi_2,\phi_3)$ be a set of orientation angles. By replacing
    $\phi_j$ with $\phi_j - \phi_1$ if necessary, we may assume without loss of generality that $\phi_1 = 0$, in which case the total surface energy can be written as
    \begin{equation}
        E(\mathcal{M},\bfphi) = \sigma(\phi_3)\mathrm{Length}(\q_1) + \sigma(-\phi_2)\mathrm{Length}(\q_2) + \sigma(\phi_2 - \phi_3)\mathrm{Length}(\q_3).
    \end{equation}

    Consider the auxiliary energy function $F:\R^4 \to \R$ defined by
    \begin{equation}
    \begin{split}
        F(a,b,r,s) &\coloneqq \sigma(s)|(a,b)-\x_1|  +  \sigma(-r)|(a,b)-\x_2|  +  \sigma(r-s)|(a,b)-\x_3|\\
        &= \sigma(2\pi - s)|(a,b)-\x_1|  +  \sigma(r)|(a,b)-\x_2|  +  \sigma(s-r)|(a,b)-\x_3|
    \end{split}
    \end{equation}
    Let $(a_\mathcal{M}, b_\mathcal{M}) = \q_1(0)$ be the coordinates of the triple junction of $\mathcal{M}$. Observe that
    \begin{equation}\label{eq:F1}
        E(\mathcal{M},\bfphi) \geq F(a_\mathcal{M}, b_\mathcal{M},\phi_2,\phi_3)
    \end{equation}
    and
    \begin{equation}\label{eq:F2}
        E(\mathcal{N},\bftheta) = F(0,0,\theta_2,\theta_3).
    \end{equation}

    The gradient and Hessian of $F$ are given by
    \begin{equation}
    \begin{split}
        \nabla F(a,b,r,s) &= \begin{pmatrix}
            \dfrac{a(r^2+c)}{|(a,b)-\x_2|} + \dfrac{(a-\sqrt{3}/2)(c+(2\pi -s)^2)}{|(a,b)-\x_1|} + \dfrac{(a+\sqrt{3}/2)(c+(r-s)^2)}{|(a,b)-\x_3|}\\
            \dfrac{(b-1)(r^2+c)}{|(a,b)-\x_2|} + \dfrac{(b+1/2)(c+(2\pi -s)^2)}{|(a,b)-\x_1|} + \dfrac{(b+1/2)(c+(r-s)^2)}{|(a,b)-\x_3|}\\
            2r|(a,b)-\x_2| + 2(r-s)|(a,b)-\x_3|\\
            -2(2\pi-s)|(a,b)-\x_1| - 2(r-s)|(a,b)-\x_3|
        \end{pmatrix}
    \end{split}
    \end{equation}
    and thus
    \begin{equation}
        \nabla F(0,0,\theta_2,\theta_3) = \begin{pmatrix}
            0\\0\\0\\0
        \end{pmatrix}.
    \end{equation}
    Moreover, we compute the Hessian matrix
    \begin{equation}
    \begin{split}
        \nabla^2 F(0,0,\theta_2,\theta_3) &= \begin{pmatrix}
            \frac{3c}{2} + \frac{2\pi^2}{3} & 0 & -\frac{2\pi}{\sqrt{3}} & \frac{4\pi}{\sqrt{3}}\\
            0 & \frac{3c}{2} + \frac{2\pi^2}{3} & -2\pi & 0\\
            -\frac{2\pi}{\sqrt{3}} & -2\pi & 4 & -2\\
            \frac{4\pi}{\sqrt{3}} & 0 & -2 & 4
        \end{pmatrix}
    \end{split}
    \end{equation}
    which has eigenvalues
    \begin{equation}
    \begin{split}
        \lambda_1(c) &= \frac{1}{12}\bigg(9c + 12 + 4\pi^2 - \sqrt{81c^2 + 72c(\pi^2-3)+16(\pi^4+18\pi^2+9)}\bigg),\\
        \lambda_2(c) &= \frac{1}{12}\bigg(9c + 12 + 4\pi^2 + \sqrt{81c^2 + 72c(\pi^2-3)+16(\pi^4+18\pi^2+9)}\bigg),\\
        \lambda_3(c) &= \frac{1}{12}\bigg( 9c + 36 + 4\pi^2 - \sqrt{81c^2 + 72c(\pi^2-9)+16(\pi^4 + 54\pi^2 + 81)} \bigg),\\
        \lambda_4(c)  &= \frac{1}{12}\bigg( 9c + 36 + 4\pi^2 + \sqrt{81c^2 + 72c(\pi^2-9)+16(\pi^4 + 54\pi^2 + 81)} \bigg).
    \end{split}
    \end{equation}
    Thus, for all sufficiently large $c\geq \overline{c}$ we have $\lambda_j(c) > 0$ for all $j=1,2,3,4$ and we conclude that $(0,0,\theta_2,\theta_3)$ is a strict local minimizer of $F$. The local minimality assertion of the Proposition then follows from (\ref{eq:F1}-\ref{eq:F2}).

    To prove stability, let $\varepsilon>0$. Let $((\q_j(x,t))_{j=1}^3 , (\phi_j(t))_{j=1}^3)$, $(x,t)\in[0,T) \times [0,1]$ be a solution to the parametric system with initial data $((\q_j^0(x))_{j=1}^3 , (\phi_j^0)_{j=1}^3)$ satisfying
    \begin{equation}\label{eq:initial_angle}
        |(\phi_{j-1}^0 - \phi_j^0) - (\theta_{j-1} - \theta_j)|<\delta, \qquad j=1,2,3, 
    \end{equation}
    \begin{equation}\label{eq:initial_length}
        |\mathrm{Length}(\q_j^0) - \mathrm{Length}(\p_j)|<\delta, \qquad j=1,2,3,
    \end{equation}
    \begin{equation}\label{eq:initial_tj}
       |\q_j^0(0)|<\delta, \qquad j=1,2,3,
    \end{equation}
    where $\delta>0$ will be determined later. Let $\mathcal{M}(t)$ denote the evolving network composed of the curves $\q_j$ and $(a(t),b(t))$ denote the coordinates of the triple junction at time $t$. Taking into account the translational invariance of the total surface energy $E$ with respect to $\bfphi$, we introduce new variables 
    \begin{equation}
        \Tilde{\bfphi} = (\Tilde{\phi}_1(t) ,\Tilde{\phi}_2(t),\Tilde{\phi}_3(t)) \coloneqq (\phi_1(t) - \phi_1(t), \phi_2(t) - \phi_1(t), \phi_3(t) - \phi_1(t))
    \end{equation} 
    Note that the energy remains unchanged as
    \begin{equation}
        E(\mathcal{M}(t),\Tilde{\bfphi}(t)) = E(\mathcal{M}(t),\bfphi(t))
    \end{equation}
    for all $t\in[0,T)$, but $\Tilde{\phi}_j(t)$ may satisfy an equation slightly different from (\ref{eq:rotation}).

    Let 
    \begin{equation}
        E^* \coloneqq E(\mathcal{N},\bftheta) = 3\sigma\bigg(\frac{2\pi}{3}\bigg).
    \end{equation}
    By (\ref{eq:initial_angle}-\ref{eq:initial_tj}),
    \begin{equation}
        E^0 \coloneqq E(\mathcal{M}(0),\bfphi(0)) \leq E^* + C_1(\delta),
    \end{equation}
    where $C_1(\delta) \to 0$ as $\delta\to 0$. 

    Let $B_{\varepsilon/4}(0,0,\theta_2,\theta_3) \subset \R^4$ be the ball of radius $\varepsilon/4$ centered at $(0,0,\theta_2,\theta_3)$. Since $(0,0,\theta_2,\theta_3)$ is a strict local minimum of $F$, we assume that $\varepsilon$ is sufficiently small so that 
    \begin{equation}
        F(a,b,r,s) > F(0,0,\theta_2,\theta_3) = E^*
    \end{equation}
    for all $(a,b,r,s) \in \overline{B}_{\varepsilon/4}(0,0,\theta_2,\theta_3) \setminus \{(0,0,\theta_2,\theta_3)\}$. Let
    \begin{equation}
        m(\varepsilon) \coloneqq \min_{\partial B_{\varepsilon/4}(0,0,\theta_2,\theta_3) }F(\cdot)
    \end{equation}
    and pick $r(\varepsilon)<\varepsilon/4$ sufficiently small so that 
    \begin{equation}
        \max_{\overline{B}_{r(\varepsilon)}(0 ,0,\theta_2,\theta_3)} F(\cdot) < \frac{m(\varepsilon)}{2}.
    \end{equation}
    Since the total surface energy $E$ is decreasing along the evolution, we get
    \begin{equation}
        F(a(t),b(t),\Tilde{\phi}_2(t),\Tilde{\phi}_3(t)) \leq E(\mathcal{M}(t),\Tilde{\bfphi}(t)) \leq E^0 \leq E^* + C_1(\delta) \leq F(a(0),b(0),\Tilde{\phi}_2^0,\Tilde{\phi}_3^0) + C_2(\delta),
    \end{equation}
    where $C_2(\delta)\to 0$ as $\delta \to 0$ and the last inequality follows from (\ref{eq:initial_angle}-\ref{eq:initial_tj}). Thus, if the initial data satisfies 
    \begin{equation}
        (a(0),b(0),\Tilde{\phi}_2^0, \Tilde{\phi}_3^0) \in B_{r(\varepsilon)}(0,0,\theta_2,\theta_3),
    \end{equation}
    then
    \begin{equation}
        F(a(t),b(t), \Tilde{\phi}_2(t),\Tilde{\phi}_3(t)) \leq \frac{m(\varepsilon)}{2} + C_2(\delta) < m(\varepsilon)
    \end{equation}
    as long as $C_2(\delta)<\frac{m(\varepsilon)}{2}$ and we have
    \begin{equation}
        (a(t),b(t),\Tilde{\phi}_2(t), \Tilde{\phi}_3(t)) \in B_{\varepsilon/4} (0,0,\theta_2,\theta_3) \qquad \forall t\in[0,T),
    \end{equation}
    which means that the triple junction and the misorientation angles remain close to those of the stationary solution at all times. This, together with the fact that $E(\mathcal{M}(t),\bfphi(t))\leq E^*+C_1(\delta)$, imply the stability assertion.

\end{proof}

Returning to differences from \cite{ep1,ep2}, we recall there the class of surface tension functions $\sigma$ considered are {\em convex} functions of the misorientation angle.
As we have seen, when the $2\pi$-periodic nature of the orientation angles are taken into account, this convex dependence can translate into the stability of the nontrivial solution (\ref{eq:stationary1}-\ref{eq:stationary2}); hence, with periodic orientation angles, not only are there non-trivial equilibrium solutions to the model of \cite{ep1,ep2}, but some are in fact stable.
In fact, we have seen that these solutions are stable also for our more general model (\ref{pde}) that allows curved interfaces, in the sense described in Proposition \ref{prop:stability}.

However, it is worth noting that there is ample reason to be wary of convex dependence of surface tension on misorientation: For one, as already mentioned, the well known model of Read \& Shockley \cite{read_shockley} prescribes a specific, concave dependence on misorientation.
Secondly, concavity of $\sigma$ (together with $\sigma(0)=0$), is equivalent to {\em subadditivity}  of the surface tension, which in turn ensures the triangle inequality (\ref{eq:triangle}) and hence well-posedness of more elaborate models of microstructure that allow nucleation.
We note that in the simplified ODE model of \cite{ep1,ep2}, it can be shown (via linear stability analysis) that the equilibrium solution $(\mathcal{N},\bftheta)$ is in fact unstable, when the surface tension is $2\pi-$periodic and is given by the Read-Shockley model
\begin{equation}
    \sigma(\theta) = A\theta(B-\ln(\theta))
\end{equation}
for $\theta\in[0,\pi]$.

\newpage

\bibliographystyle{plain}
\bibliography{references}

\newpage

\appendix
\section{Appendix}
Here, we collect some simple facts and standard results that are used throughout the paper.

\begin{lemma}\label{lemma-zero}
    Let $0<\beta<1$ and $f\in C^\beta([0,T];\R^n)$ be such that $f(0) = 0$. Then
    \begin{equation}
        |f|^{(0)} \leq T^\beta\langle f\rangle^{(\beta)}.
    \end{equation}
\end{lemma}

\begin{proof}
    \begin{equation}
    \begin{split}
        \sup_{t\in[0,T]}|f(t)| &= \sup_{t\in[0,T]}|f(t) - f(0)|\\
        &\leq \sup_{t,s\in[0,T], \; t\neq s}|f(t)-f(s)|\\
        &=\sup_{t,s\in[0,T],\;t\neq s} \bigg\{\frac{|f(t)-f(s)||t-s|^\beta}{|t-s|^\beta}\bigg\}\\
        &\leq T^\beta\langle f\rangle^{(\beta)}
    \end{split}
    \end{equation}
\end{proof}

\begin{lemma}\label{lemma-lipschitz}
    Let $U\subset \R^m$ be a bounded domain and $g\in C^\alpha(U;\R^n)$ with $0<\alpha<1$. Let $F:range(g)\to\R$ be a Lipschitz continuous function with Lipschitz constant $Lip(F)$. Then for any $0<\beta\leq\alpha$,
    \begin{equation}
        \langle F\circ g\rangle^{(\beta)} \leq Lip(F)\;diam(U)^{\alpha-\beta} \langle f\rangle^{(\alpha)}.
    \end{equation}
\end{lemma}

\begin{proof}
    \begin{equation}
    \begin{split}
        \langle F\circ g\rangle^{(\beta)} &= \sup_{x,y\in U, \; x\neq y} \bigg\{ \frac{|F(f(x)) - F(f(y))|}{|x-y|^\beta} \bigg\}\\
        &\leq \sup_{x,y\in U, \; x\neq y} \bigg\{ \frac{Lip(F)|f(x)-f(y)||x-y|^{\alpha-\beta}}{|x-y|^\beta |x-y|^{\alpha-\beta}} \bigg\}\\
        &\leq Lip(F)\;diam(U)^{\alpha-\beta} \langle f\rangle^{(\alpha)}
    \end{split}
    \end{equation}
\end{proof}

\begin{lemma}\label{lemma-difference}
    Let $f,g\in C^{\beta}([0,T])$ for some $0<\beta<1$. Suppose $range(f),\; range(g) \subset [a,b]$ with $-\infty<a<b<+\infty$. Let $\sigma \in C^1([a,b])$ be such that $\sigma'$ is Lipschitz, with Lipschitz constant $Lip(\sigma')$.  
    
    Then 
    \begin{equation}
        \langle \sigma\circ f - \sigma\circ g \rangle^{(\beta)} \leq C_\sigma \bigg( \langle f-g\rangle^{(\beta)} + \langle g\rangle^{(\beta)}|f-g|^{(0)} \bigg).
    \end{equation}
\end{lemma}

\begin{proof}
    By definition of the H\"{o}lder seminorm,
    \begin{equation}
    \begin{split}
        \langle \sigma\circ f - \sigma\circ g \rangle^{(\beta)} &= \sup_{s\neq t}\bigg\{ \frac{1}{|s-t|^\beta}\big| [\sigma(f(s)) - \sigma(g(s))]  -  [\sigma(f(t) - \sigma(g(t))]  \big|\bigg\}\\
        &= \sup_{s\neq t}\bigg\{ \frac{1}{|s-t|^\beta}\big| h(f(s),g(s)) - h(f(t),g(t))  \big|\bigg\}\\
    \end{split}
    \end{equation}
    where $h:[a,b]^2 \to \R$ is defined by
    \begin{equation}
        h(a,b) = \sigma(a) - \sigma(b).
    \end{equation}
    By the mean value theorem, there exists $c\in[0,1]$ such that
    \begin{equation}
    \begin{split}
        h(f(s),g(s)) - h(f(t),g(t)) &= (\nabla h)\bigg((1-c)(f(t),g(t)) + c(f(s),g(s))\bigg)\cdot \bigg( (f(s),g(s)) - (f(t),g(t)) \bigg)\\
        &= \sigma'\bigg( (1-c)f(t) + cf(s) \bigg)[f(s)-f(t)-g(s)+g(t)] \\
        &+ [g(s)-g(t)]\bigg[ \sigma'\bigg( (1-c)f(t) + cf(s) \bigg) - \sigma'\bigg( (1-c)g(t) + cg(s) \bigg)\bigg].
    \end{split}
    \end{equation}
    Thus,
    \begin{equation}
    \begin{split}
        | h(f(s),g(s)) - h(f(t),g(t))  \big| &\leq |\sigma'|^{(0)} |f(s)-f(t)-g(s)+g(t)| \\
        &+ |g(s)-g(t)|Lip(\sigma') |(1-c)(f(t)-g(t)) + c(f(s)-g(s))| \\
        &\leq C_\sigma \bigg(|f(s)-f(t)-g(s)+g(t)| + |g(s)-g(t)||f-g|^{(0)} \bigg)
    \end{split}
    \end{equation}
    and so
    \begin{equation}
    \begin{split}
        \langle \sigma\circ f - \sigma\circ g \rangle^{(\beta)} 
        &= \sup_{s\neq t}\bigg\{ \frac{1}{|s-t|^\beta}\big| h(f(s),g(s)) - h(f(t),g(t))  \big|\bigg\}\\
        &\leq \sup_{s\neq t}\bigg\{ \frac{1}{|s-t|^\beta} C_\sigma \bigg(|f(s)-f(t)-g(s)+g(t)| + |g(s)-g(t)||f-g|^{(0)} \bigg) \bigg\}\\
        &\leq C_\sigma \bigg( \langle f-g\rangle^{(\beta)} + \langle g\rangle^{(\beta)}|f-g|^{(0)} \bigg).
    \end{split}
    \end{equation}
    
\end{proof}

\begin{lemma}\label{lem:diffeo}
    Let $a,b \in \R$. Then there exists a smooth diffeomorphism $\varphi:[0,1] \to [0,1]$ satisfying
    \begin{enumerate}[label=(\roman*)]
        \item $\varphi(0) = 0$, $\varphi(1) = 1$,
        \item $\varphi'(0) = \varphi'(1) = 1$,
        \item $\varphi'(x) > 0$ for all $x\in(0,1)$,
        \item $\varphi''(0) = a$, $\varphi''(1) = b$.
    \end{enumerate}
\end{lemma}

\begin{proof}
    Let $f:[0,1] \to (0,+\infty)$ be a smooth function with $f(0) = f(1) = 1$, $f'(0) = a$, $f'(1) = b$ and $\int_0^1 f(x) \;dx = 1$. This can be done easily by perturbing the constant function 1 with small bump functions. Then, define
    \begin{equation}
        \varphi(x) = \int_0^x f(s) \; ds.
    \end{equation}
\end{proof}

\begin{theorem}[Comparison Principle]\label{thm:strongmaxprinciple}

For $j=1,2$, let $I_j$ be either $[0,1]$ or $S^1$ and let $\Gamma_j:I_j \times [0,T] \to \R^2$ be two immersed curves evolving by curvature motion. (They could be open or closed curves.) Suppose that $\Gamma_j$ are sufficiently smooth, say, $\Gamma_j \in C^{2+\alpha,1+\alpha/2}(I_j\times[0,T])$. Assume that initially
\begin{equation}
    \mathrm{dist}( \Gamma_1(\cdot,0) , \Gamma_2(\cdot,0) ) > 0.
\end{equation}
Assume that the two curves intersect for the first time at $t_* \in (0,T)$, i.e. $\Gamma_1(x_1,t_*) = \Gamma_2(x_2,t_*)$ and $\mathrm{dist}( \Gamma_1(\cdot,t) , \Gamma_2(\cdot,t) ) > 0$ for all $t \in [0,t_*)$. Then we must have either $x_1 \in \{0,1\}$ or $x_2 \in \{0,1\}$. In other words, the two curves cannot meet for the first time in an interior-interior setting.
\end{theorem}

\begin{proof}
    Such comparison principle (or maximum principle) results are well-known and can be found in references such as \cite{Ecker,murray-weinberger}. For completeness, we will briefly recall the proof. Suppose for a contradiction that there exists a time $t_* \in (0,T)$ such that the two curves meet for the first time at interior points, i.e. $\Gamma_1(x_1,t_*) = \Gamma_2(x_2,t_*)$ with $x_1,x_2 \in (0,1)$ (or $S^1$) and $\mathrm{dist}( \Gamma_1(\cdot,t) , \Gamma_2(\cdot,t) ) > 0$ for all $t \in [0,t_*)$. Then, locally we can write the curves $\Gamma_j$ as evolving graphs:
    \begin{equation}
        u_j:[-a,a]\times[t_*-\varepsilon,t_*] \to \R, \qquad j=1,2
    \end{equation}
    satisfying the parabolic equations
    \begin{equation}
            u_{jt} = \frac{u_{jxx}}{1+u_{jx}^2}
    \end{equation}
    with
    \begin{equation}
        \begin{split}
            u_1(0,t_*) &= u_2(0,t_*)\\
            u_{1x}(0,t_*) &=  u_{2x}(0,t_*) = 0\\
            u_1(x,t) > u_2(x,t) &\text{ for all } (x,t) \in [-a,a]\times[t_*-\varepsilon,t_*).
        \end{split}
    \end{equation}

    Let $w(x,t) = u_2(x,t) - u_1(x,t)$, which satisfies the uniformly parabolic equation
    \begin{equation}
        w_t = a(x,t)w_{xx} + b(x,t)w_x
    \end{equation}
    with
    \begin{equation}
    \begin{split}
        a(x,t) &= \frac{1}{1 + u_{2x}^2}\\
        b(x,t) &= -\frac{u_{1xx}(u_{1x}+u_{2x})}{(1+u_{1x}^2)(1+u_{2x}^2)}.
    \end{split}
    \end{equation}

    However, the fact that
    \begin{equation}
        w(0,t_*) = 0 = \sup_{[-a,a]\times[t_*-\varepsilon,t_*]}w(x,t)
    \end{equation}
    violates the strong maximum principle.

\end{proof}

\end{document}